%% file: chi_matchsticks.tex
\documentclass[11pt]{amsart}
\usepackage{mathptmx,microtype}
\usepackage[T1]{fontenc}
\usepackage{amsmath, amscd, amsthm} 
\usepackage{amssymb, amsfonts} 
\usepackage{stmaryrd}
\usepackage{adjustbox}
\usepackage{enumerate}
\usepackage[svgnames]{xcolor} 
\usepackage{color}
\usepackage[margin=1.5in]{geometry}
\usepackage{cancel}
\usepackage{mathrsfs}
\usepackage{booktabs}
\usepackage{aliascnt}
\usepackage{mathrsfs}
\usepackage{hyperref}
\usepackage{slashed}
\usepackage{longtable}
\usepackage{mathtools}
 \definecolor{dark-red}{rgb}{0.4,0.15,0.15}
\hypersetup{
    colorlinks, linkcolor=dark-red,
    citecolor=DarkBlue, urlcolor=MediumBlue
}
\usepackage{mathrsfs}
\usepackage[
textwidth=3cm,
textsize=small,
colorinlistoftodos]
{todonotes}

\usepackage{tikz}
\usepackage{tikz-cd}
\usetikzlibrary{automata,calc,decorations.pathreplacing} 

\usepackage{multicol}

\newcommand{\NN}{\mathbb{N}}

\DeclareMathOperator{\Tr}{Tr}
\newcommand{\End}{\operatorname{End}}
\DeclareMathOperator{\Cat}{Cat}

\newcommand{\Sub}{\operatorname{Sub}}

\usepackage[T1]{fontenc}

\numberwithin{equation}{section} 

\theoremstyle{plain}

\newaliascnt{theorem}{equation}  
\newtheorem{theorem}[theorem]{Theorem}  
\aliascntresetthe{theorem}

\newtheorem*{theorem*}{Theorem}

\newaliascnt{dodeca}{equation}  
  
\aliascntresetthe{dodeca}

 \theoremstyle{definition}

\newaliascnt{prop}{equation}  
\newtheorem{prop}[prop]{Proposition}
\aliascntresetthe{prop}

\newaliascnt{lemma}{equation}  
\newtheorem{lemma}[lemma]{Lemma}
\aliascntresetthe{lemma}

\newaliascnt{corollary}{equation}  
\newtheorem{corollary}[corollary]{Corollary}
\aliascntresetthe{corollary}

\newaliascnt{claim}{equation}  

\aliascntresetthe{claim}

\newaliascnt{conjecture}{equation}  

\aliascntresetthe{conjecture}

\newaliascnt{question}{equation}  

\aliascntresetthe{question}

\newaliascnt{defn}{equation}  
\newtheorem{defn}[defn]{Definition}
\aliascntresetthe{defn}

\newaliascnt{example}{equation}  
\newtheorem{example}[example]{Example}
\aliascntresetthe{example}

\theoremstyle{remark}

\newaliascnt{remark}{equation}  
\newtheorem{remark}[remark]{Remark}
\aliascntresetthe{remark}

\newaliascnt{convention}{equation}  

\aliascntresetthe{convention}

\theoremstyle{plain}

\begin{document}
\title[Transfer systems for rank two elementary Abelian groups]{Transfer systems for rank two elementary Abelian groups: characteristic functions and matchstick games}

\author[Bao]{Linus Bao}
\author[Hazel]{Christy Hazel}
\author[Karkos]{Tia Karkos}
\author[Kessler]{Alice Kessler}
\author[Nicolas]{Austin Nicolas}
\author[Ormsby]{Kyle Ormsby}
\author[Park]{Jeremie Park}
\author[Schleff]{Cait Schleff}
\author[Tilton]{Scotty Tilton}

\begin{abstract}
We prove that Hill's characteristic function $\chi$ for transfer systems on a lattice $P$ surjects onto interior operators for $P$. Moreover, the fibers of $\chi$ have unique maxima which are exactly the saturated transfer systems. In order to apply this theorem in examples relevant to equivariant homotopy theory, we develop the theory of saturated transfer systems on modular lattices, ultimately producing a ``matchstick game'' that puts saturated transfer systems in bijection with certain structured subsets of covering relations. After an interlude developing a recursion for transfer systems on certain combinations of bounded posets, we apply these results to determine the full lattice of transfer systems for rank two elementary Abelian groups.
\end{abstract}

\maketitle

\setcounter{tocdepth}{1}
\tableofcontents

\section{Introduction}\label{sec:intro}
\subsection{Motivation, context, and results}
In topology, there is a unique $E_\infty$ operad up to homotopy. Moreover, localization of ring spectra preserves $E_\infty$-structures. These simple facts underpin derived --- or, more precisely, spectral --- algebraic geometry. Indeed, it is very hard to mimic algebraic geometry without a well-behaved theory of localization.

Fix a finite group $G$. In $G$-equivariant topology, Blumberg and Hill \cite{BlumbergHill} have introduced the notion of a $G$-$N_\infty$ operad. These are operads in $G$-spaces encoding homotopy coherent ring structure and compatible families of \emph{norms}: multiplicative wrong-way maps $N_H^G R\to R$ where $H$ is a subgroup of $G$ and $N_H^G$ is left adjoint to the forgetful functor from $G$-spectra to $H$-spectra. These norms are a fundamental feature of $G$-ring spectra and are essential to the computational work following from the Hill--Hopkins--Ravenel resolution of the Kervaire invariant one problem \cite{hhr}.

Unlike the situation in classical topology, there are \emph{many} homotopically distinct $G$-$N_\infty$ operads and localization of $G$-ring spectra does \emph{not} preserve $N_\infty$-structure. This makes the prospect of equivariant derived algebraic geometry both daunting and excitingly new.

Given this challenge, there has been significant recent interest in uncovering the structure of the category of $G$-$N_\infty$ operads \cite{gw18,Rubin,bp21,bbr,fooqw,hmoo,bmo:CD,lifting}. The homotopy category of $G$-$N_\infty$ operads is equivalent to the category of so-called $G$-transfer systems (see \autoref{defn:tr}). These are combinatorial structures on the subgroup lattice of $G$ which can be investigated with traditional enumerative, algebraic, and asymptotic methods.

In this work, we determine the structure of the partially ordered set (in fact, lattice) of $G$-transfer systems for $G$ a rank $2$ elementary Abelian group (\emph{i.e.}, $G\cong C_p\times C_p$ for $C_p$ a cyclic group of prime order $p$):

\begin{theorem*}[\autoref{thm:ranktwo}]
Let $p$ be a prime number. There are exactly $2^{p+2}+p+1$ transfer systems for $C_p\times C_p$, and the lattice of transfer systems has an explicit decomposition in terms of two Boolean lattices on $p+1$ elements along with $p+1$ intermediate elements (see \autoref{prop:interatedfusion} for lattice structure details).
\end{theorem*}

This theorem represents only the fourth infinite class of groups whose transfer systems have been determined in either closed or recursive form. Previously, Balchin--Barnes--Roitzheim \cite{bbr} proved that $\Tr(C_{p^n})$ is isomorphic to the Tamari lattice with $\Cat(n+1) = \frac{1}{n+2}\binom{2n+2}{n+1}$ elements. In \cite{bmo:CD}, Balchin--MacBrough--Ormsby give a recursion for transfer systems on dihedral groups of order $2p^n$ and cyclic groups of order $qp^n$, $p\ne q$ primes; they do not give a closed formula for the number of transfer systems, and they do not investigate the lattice structure.

While it is likely possible to prove the above theorem from first principles, we take a roundabout route that allows us to expose and develop several important structural properties of transfer systems in general. These split into three categories: characteristic functions, saturated covers on modular lattices, and categorical properties of transfer systems.

The idea of a characteristic function for lattices is due to Hill, and may be viewed as a relativization of the notion of minimal fibrancy exploited in \cite{bmo:CD}. Let $P$ be a finite lattice. There is a simple extension of the notion of a transfer system from a subgroup lattice to an arbitrary lattice (see \autoref{defn:tr}). In our formulation, the characteristic function takes the form of an antitone (order-reversing) map
\[
  \chi\colon \Tr(P)\longrightarrow \End(P)
\]
where $\Tr(P)$ is the lattice of transfer systems on $P$, and $\End(P)$ is monotone endomorphisms of $P$. (See \autoref{sec:chi} for details on the construction.)

Let $\End^\circ(P)\subseteq \End(P)$ denote the set of interior operators on $P$, \emph{i.e.}, idempotent, contractive, monotone maps $P\to P$. We prove the following:

\begin{theorem*}[\autoref{thm:maxsat} and \autoref{thm:minchi}]
Let $P$ be a finite lattice. Hill's characteristic function $\chi$ has image $\End^\circ(P)$, and each fiber of $\chi$ is an interval $[R,R'] = \{R''\in \Tr(P)\mid R\le R''\le R'\}$ in $\Tr(P)$. The set of maximal elements of fibers of $\chi$ is exactly the set of saturated transfer systems.
\end{theorem*}

Here \emph{saturated transfer systems} (\autoref{defn:sat}) are transfer systems satisfying an additional 2-out-of-3 condition. They have appeared previously in the literature because the transfer system associated with any linear isometries operad is saturated \cite[Proposition 5.1]{Rubin}. (It is not the case, though, that every saturated transfer system is realized by some linear isometries operad.) The above theorem highlights the important structural role that saturated transfer systems play independent of their affiliation with linear isometries.

The above theorem places saturated transfer systems in bijection with interior operators, well-known objects of study from order theory which are cryptomorphically equivalent to several other structures (submonoids of $(P,\vee)$ and comonads on $P$, to name a couple). It is the authors' hope that the transfer system perspective on interior operators will lead to new insights in this classical subject.

Saturated transfer systems have been enumerated on rectangular lattices by Hafeez--Marcus--Ormsby--Osorno \cite{hmoo}. They reduce their study to a ``matchstick game'' consisting of collections of covering relations satisfying certain rules. We generalize this method to all modular lattices, \emph{i.e.}, lattices satisfying the modular law
\[
  a\le b\implies a\vee(x\wedge b) = (a\vee x)\wedge b.
\]
Such lattices are of particular interest since subgroup lattices of Abelian groups are modular. In this setting, we call ``matchstick games'' \emph{saturated covers} (\autoref{defn:satcov}) and prove that they are in bijection with saturated transfer systems (\autoref{thm:matchstick}).

Our final main contribution in this paper is to study some categorical properties of transfer systems. Given finite lattices $P,Q$, we may form their \emph{fusion} $P*Q$ (\autoref{defn:fusion}) which, loosely speaking, takes a disjoint union of posets and then glues extremal elements together. \autoref{thm:fusion} enumerates transfer systems on $P*Q$ in terms of (subposets of) $P$ and $Q$. If $[2] = \{0<1<2\}$, then the sugroup lattice of $C_p\times C_p$ is isomorphic to the iterated fusion $[2]^{*(p+1)}$ of $[2]$ with itself, and this leads to the enumeration of \autoref{thm:ranktwo}.

\subsection{Transfer systems}

For reference throughout the rest of this document, we now recall the basic definitions and theorems on which we build. The reader familiar with transfer systems may easily skip this subsection and refer back if needed. We assume throughout that the reader is familiar with the basics of partially ordered sets and lattices; the textbook \cite{davey_priestley_2002} may act as a handy reference.

\begin{defn}\label{defn:tr}
Let $(P,\le)$ be a finite lattice. A \emph{transfer system} on $P$ is a partial order $R$ on $P$ that \emph{refines} $\le$ (\emph{i.e.}, $x~R~y\implies x\le y$) and is closed under restriction:
\[
  x~R~z\text{ and }y\le z\implies (x\wedge y)~R~y.
\]
If $G$ is a finite group, then a \emph{$G$-transfer system} is a transfer system on $\Sub(G)$, the subgroup lattice of $G$, which is further closed under conjugation. We write $\Tr(P)$ for the collection of transfer systems on $P$, and $\Tr(G)$ for the collection of $G$-transfer systems.
\end{defn}

In categorical language, we see that a transfer system is the same thing as a wide subcategory of (the category induced by) $P$ that is closed under pullbacks.

There is a natural refinement order on transfer systems where $R\le R'$ if and only if $x~R~y\implies x~R'~y$. Under this order, $(\Tr(P),\le)$ is a finite poset that admits meets: if $R,R'\in \Tr(P)$, then $R\wedge R'$ is given by intersecting the relations in $R$ and $R'$ so that $x~(R\wedge R')~y$ if and only if both $x~R~y$ and $x~R'~y$. Since $\Tr(P)$ is a finite poset admitting meets that has a greatest element (namely, the original partial order on $P$), $\Tr(P)$ also admits joins:
\[
  R\vee R' = \bigwedge_{R''\ge R,R'} R''.
\]
This yields the following proposition.
\begin{prop}
If $P$ is a finite lattice, then $\Tr(P)$ is a finite lattice under refinement. The same is true for $\Tr(G)$ when $G$ is a finite group.\hfill\qedsymbol
\end{prop}

Our primary interest in transfer systems comes from their role in the theory of $N_\infty$ operads. These are operads $\mathscr O$ in $G$-spaces introduced and defined in \cite{BlumbergHill} and having the feature that for $H\le G\times \mathfrak{S}_n$, $\mathscr{O}(n)^H$ is either empty or contractible. Their basics are reviewed in \cite[\S 3]{fooqw}, and their connection to transfer systems is as follows:

\begin{theorem}[\cite{BlumbergHill,Rubin}]
Let $G$ be a finite group. There is a functor from the category of $G$-$N_\infty$ operads to (the category induced by) the lattice $\Tr(G)$ that induces an equivalence between the homotopy category of $G$-$N_\infty$ operads and $\Tr(G)$. The $G$-transfer system $R_{\mathscr O}$ induced by a $G$-$N_\infty$-operad $\mathscr{O}$ is such that $H~R_{\mathscr O}~K$ if and only if $\mathscr O$-algebras admit $K/H$-norms.\hfill\qedsymbol
\end{theorem}

It is also the case that transfer systems on a finite lattice $P$ are equivalent to weak factorization systems on (the category induced by) $P$. See \cite{fooqw} for details on this surprising connection between transfer systems and abstract homotopy theory. The article \cite{boor} leverages this to enumerate model structures on categories induced by finite total orders.

In order to graphically represent transfer systems, we adopt the convention of displaying \emph{all} of the relations in the transfer system with upward oriented (but unmarked) edges. This is quite different from the standard Hasse diagram presentation of posets, which only displays covering (\emph{i.e.}, minimal) relations. This is standard practice when discussing transfer systems as it is otherwise challenging to visually inspect the restriction condition. In practice, we first arrange $P$ in a Hasse diagram, then erase the covering relations in the diagram, and finally draw in all of the non-reflexive relations present in $R\in \Tr(P)$. For instance, in \autoref{fig:tamari} we see the transfer systems on $[2] = \{0<1<2\}$ arranged in their pentagonal Hasse diagram under refinement. (See \cite{bbr} for why $\Tr([n])$ is always a Tamari lattice.)

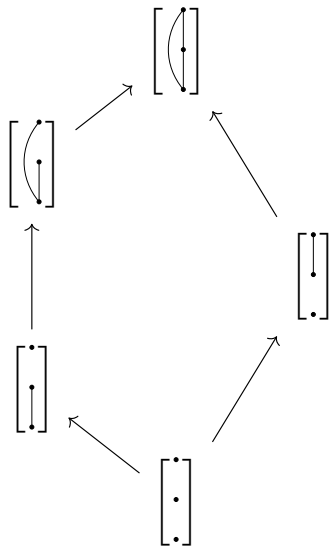
\begin{figure}
\input{Tamari2}
\caption{Transfer systems for the lattice $[2]$.}
\label{fig:tamari}
\end{figure}

At various points in our exposition, it will be important to study the transfer system generated by a subset of binary relations $Q\subseteq {\le}\subseteq P\times P$.

\begin{defn}\label{defn:gen}
Let $(P,\le)$ be a lattice and suppose $Q\subseteq {\le}$ is a collection of relations refining $\le$. The \emph{transfer system generated} by $Q$, denoted $\langle Q\rangle$, is the minimal transfer system containing $Q$:
\[
  \langle Q\rangle = \bigwedge_{\substack{R\in \Tr(P)\\ R\supseteq Q}} R.
\]
\end{defn}

The following theorem (due to Rubin) gives us a great deal of control over $\langle Q\rangle$ in terms of $Q$.

\begin{theorem}[{\cite[Theorem A.2]{Rubin}}]\label{prop:gen}
If $(P,\le)$ is a finite lattice and $Q\subseteq {\le}$ is a subset of relations of $P$, then $\langle Q\rangle$ may be constructed from $Q$ by (1) closing under reflexivity, then (2) closing under restriction, and finally (3) closing under transitivity.
\end{theorem}

We will also be interested in saturated transfer systems, which satisfy a certain 2-out-of-3 property:

\begin{defn}\label{defn:sat}
A transfer system $R$ on a lattice $P$ is called \emph{saturated} when $x~R~y\le z$ and $x~R~z$ implies $y~R~z$.
\end{defn}

Note that of the five transfer systems for $[2]$ presented in \autoref{fig:tamari}, all but
\begin{tikzpicture}[scale = 0.7]
    \coordinate (0) at (0:0);
        \coordinate (1) at (0:{sqrt(2)/2}); 
        \coordinate (2) at ($(1) + (1)$);
        \foreach \i in {(0),(1),(2)}{
            \draw[white, fill =white] \i circle (0.1);
            \draw[black,fill = black] \i circle (0.045);
        }
        \draw[line width=0.1mm] (0)--(1);
        \draw[line width=0.1mm] (0) to[out=30,in=150] (2);
\end{tikzpicture}
are saturated.

In \autoref{sec:matchstick}, we will prove that saturated transfer systems on modular lattices are in bijection with certain collections of covering relations. We warn the reader that this will lead to an alternative graphical presentation of saturated transfer systems that only records the relevant covering relations from the underlying lattice.

We conclude this review by noting that saturated transfer systems are determined by their covering relations; this will be important in \autoref{sec:matchstick} when we contemplate transfer systems generated by saturated covers.

\begin{prop}[{\cite[Proposition 5.8]{Rubin}}]\label{prop:sat_det_cov}
Let $(P,\le)$ be a finite lattice, let $R$ be a saturated transfer system on $P$, and let $R_{cov}$ denote the covering relations of $P$ in $R$. Then $\langle R_{cov}\rangle = R$.
\end{prop}

\subsection{Outline of the paper}
In \autoref{sec:chi}, we develop Hill's theory of characteristic functions for transfer systems. Given a transfer system $R$ on a lattice $P$, $\chi^R$ is the monotone function taking $x\in P$ to $\chi^R(x) = \min \{y\in P\mid y~R~x\}$. In \autoref{thm:interior}, we show that the assignment
\[
\begin{aligned}
  \chi\colon \Tr(P)&\longrightarrow \End(P)\\
  R&\longmapsto \chi^R
\end{aligned}
\]
is an antitone function with image equal to the \emph{interior operators} on $P$ (i.e.~contractive and idempotent endomorphisms of $P$). Each nonempty fiber $\chi^{-1}\{f\}$ is in fact an interval in $\Tr(P)$ with least element easily determined by $f$, and greatest element saturated. In fact, in \autoref{thm:maxsat}, we show that the set of greatest elements of fibers of $\chi$ is equal to the set of saturated transfer systems on $P$.

Given the enhanced role for saturated transfer systems afforded by \autoref{thm:maxsat}, we turn to their study in \autoref{sec:matchstick}. Here we specialize to \emph{modular lattices} (those with no pentagonal sublattices) for two reasons: first, every subgroup lattice of an Abelian group is modular, and second, modularity allows us to generalize the ``matchstick game'' for saturated transfer systems on rectangular lattices from \cite{hmoo}. After recalling some basic facts about modular lattices, we prove \autoref{thm:matchstick} which places saturated transfer systems on a modular lattice $P$ in bijection with subsets of the covering relations of $P$ satisfying a restriction rule while avoiding a ``3-out-of-4'' configuration in covering diamonds.

In \autoref{sec:cat} we make a slight detour to study some categorical properties of transfer systems. This culminates in a general recursion (\autoref{thm:fusion}) for transfer systems on the \emph{fusion} $P*Q$ of two lattices $P,Q$; this is our shorthand for the lattice formed by identifying the minimal (resp.~maximal) elements of $P$ and $Q$, retaining all of the existing relations in $P$ and $Q$, and making each non-extremal $x\in P$, $y\in Q$ incomparable.

The iterated fusion $[2]^{*n}$  of $[2] = \{0<1<2\}$ is modular for all $n\in \NN$ and is isomorphic to the subgroup lattice of the elementary Abelian group $C_p\times C_p$ for $p$ prime and $n=p+1$. In \autoref{sec:CpCp}, we combine this observation with the techniques developed earlier in the paper to fully specify and enumerate the lattice of transfer systems for $C_p\times C_p$ (see \autoref{thm:ranktwo}).

\subsection*{Acknowledgments}
This work arose from the 2023 Electronic Computational Homotopy Theory (eCHT) Research Experience for Undergraduates, supported by NSF grant DMS--2135884; the authors thank Dan Isaksen for his support and leadership within the eCHT. K.O.~was partially supported by NSF grant DMS--2204365. The authors thank Mike Hill for generously sharing his construction of the characteristic map $\chi$ on transfer systems.

\section{Characteristic functions for transfer systems}\label{sec:chi}

Given a transfer system $R$ on a lattice $P$, we now describe its \emph{characteristic function} $\chi^R\colon P\to P$ along with the properties of the assignment $\chi\colon \Tr(P)\to \End(P)$, $R\mapsto \chi^R$. We warmly thank Mike Hill for generously sharing this construction and its basic properties with us. The results in the first part of this section are due to Hill, while \autoref{thm:interior} and \autoref{thm:maxsat} are original.

Hill's construction begins by considering the downsets of transfer systems. Given a lattice $P$, a transfer system $R$ on $P$, and an element $x$ of $P$, write
\[
  x^\downarrow_R := \{y\in P\mid y~R~x\}
\]
for the $R$-downset of $x$.

\begin{prop}\label{prop:strat}
For any finite lattice $P$, $x\in P$, and $R\in \Tr(P)$, the following properties hold:
\begin{enumerate}[(a)]
\item $x_R^\downarrow$ has a unique minimal element $n$.
\item If $n\le y\le x$ and $z~R~y$, then $n\le z$.
\end{enumerate}
\end{prop}
\begin{proof}
For (a), the restriction axiom implies that $n = \bigwedge_{y\in x_R^\downarrow} y$ works.

For (b), if $z~R~y$, then taking meet with $n$ gives $(n\wedge z)~R~((n\wedge y) = n)~R~x$. By minimality of $n$, we thus must have $n\le n\wedge z$, so in fact $n\le z$ as desired.
\end{proof}

\begin{defn}
Given a finite lattice $P$ and $R\in \Tr(P)$, the \emph{characteristic function} of $R$ is
\[
\begin{aligned}
  \chi^R\colon P&\longrightarrow P\\
  x&\longmapsto \min x^\downarrow_R.
\end{aligned}
\]
\end{defn}

By \autoref{prop:strat}(b), we see that $P$ is stratified by the subsets
\[
  \chi^R(x)^\uparrow = \{y\in P\mid \chi^R(x)\le y\}.
\]
Indeed, as $x$ varies, we get a family of nested subsets $\chi^R(x)^\uparrow$ of $P$ whose union is $P$ and such that there are no $R$-relation crossing from lower to higher strata.

The following property follows directly from definitions:
\begin{prop}\label{prop:chimono}
Given a transfer system $R$ on a finite lattice $P$, $\chi^R\colon P\to P$ is monotone.\hfill\qedsymbol
\end{prop}

Write $\End(P)$ for the collection of (monotone) endomorphisms of $P$ equipped with the pointwise partial order: $f\le g$ if and only if $f(x)\le g(x)$ for all $x\in P$. It is not difficult to check that when a transfer system $R$ refines a transfer system $R'$ (so $R\le R'$ in the canonical partial order on $\Tr(P)$), we have $\chi^{R'}\le \chi^R$. We record this fact in the following proposition:

\begin{prop}\label{prop:antitone}
The map
\[
\begin{aligned}
    \chi\colon \Tr(P)&\longrightarrow \End(P)\\
    R&\longmapsto \chi^R
\end{aligned}
\]
is an antitone map of posets.\hfill\qedsymbol
\end{prop}

We now study the image of $\chi$. It follows directly from definitions that each $f = \chi^R$ is an \emph{interior operator}, \emph{i.e.} $f$ is \emph{idempotent} ($f(f(x))=f(x)$ for all $x$) and \emph{contractive} ($f(x)\le x$ for all $x$). We write $\End^\circ(P)\subseteq \End(P)$ for the collection of interior operators on $P$. We show all interior operators are contained in the image, which requires the following lemma.

\begin{lemma}\label{lem:diagram}
Suppose $f$ is an interior operator on a poset $P$, $y,z\le x\in P$, and $f(x)\le y$. Then $f(z)\le y\wedge z$.
\end{lemma}
\begin{proof}
Notice that $f(z) \leq f(x) \leq y$ and $f(z)\leq z$, so $f(z) \leq y\wedge z$.
\end{proof}

\begin{theorem}\label{thm:interior}
The characteristic map $\chi\colon \Tr(P)\to \End(P)$ has image equal to $\End^\circ(P)$.
\end{theorem}

\begin{proof}
    We have already observed that $\chi(\Tr(P))\subseteq \End^\circ(P)$, so we now check the opposite inclusion. Let $f\in \End^\circ(P)$ and define a relation $R$ on $P$ via $y~R~x$ if either $y=x$ or $f(x)=y$. Then let $R'$ be the closure of $R$ under restriction, and finally let $R_{*}$ be the closure of $R'$ under transitivity. By \autoref{prop:gen}, we get that $R_{*}=\langle R\rangle$. It now suffices to show that $y~R_{*}~x$ implies $f(x)\leq y$; indeed, if this holds then $f(x)=\text{min }x^{\downarrow}_{R}$ and so $\chi^{R_{*}}=f$. Let $y,x\in P$ such that $y~R_{*}~x$. If $y~R~x$ then $f(x)=y$ or $f(x)\leq x =y$ as desired. So suppose that $y \mathrel{\cancel R} x$. Now we have two cases. First suppose that $y~R'~x$. Then we get that there exists $z,w\in P$ such that $y=x\wedge z$ and we have the following diagram:\\
    \begin{center}
    \begin{tikzcd}[ampersand replacement=\&]
        x\arrow[r] \&
        w \\
        y=x\wedge z \arrow[u,"R'"] \arrow[r]\&
        z. \arrow[u,"R"]
    \end{tikzcd}
    \end{center}
    Because $z~R~w$ we in fact have that $f(w)=z$. Thus by \autoref{lem:diagram}, we see that $f(x)\leq y$. Now suppose that $y\mathrel{\cancel{R'}}x$. Thus $y~R_{*}~x$ must have arisen from closure under transitivity, \emph{i.e.}, we have $y~R'~z~R'~x$ for some $z\in P$. Then using the previous cases, we know that $f(x)\leq z$, and $f(z)\leq y$. Thus because $f$ is an interior operator, we get that $f(x)=f(f(x))\leq f(z)\leq y$. Hence we see that for all $y~R_{*}~x$ that $f(x)\leq y$. Therefore $\chi(\Tr(P))=\End^\circ(P)$.
\end{proof}

Now that we know the image of $\chi$, we turn to investigating the fibers of this map. Our next lemma shows the fibers are closed under joins, which guarantees each fiber has a greatest element.
\begin{lemma}\label{lem:chijoins}
Let $P$ be a lattice and let $R,R'\in \Tr(P)$, then $\chi^{R}=\chi^{R'}$ implies $\chi^{R\vee R'}=\chi^{R}=\chi^{R'}$ (where $R\vee R'$ is the join of $R,R'$ in $\Tr(P)$).
\end{lemma}
\begin{proof}
Suppose for contradiction that $\chi^{R\vee R'}\neq \chi^R$. Since $R\leq R \vee R'$ and $\chi$ is an antitone map, it must be that $\chi^{R\vee R'}< \chi^R$. Thus there is some $x\in P$ such that $z= \text{min }x^{\downarrow}_{R\vee R'}< \text{min }x^{\downarrow}_{R}=\text{min }x^{\downarrow}_{R'}=y$. In particular, this means $z$ is $R\vee R'$ related to $x$, but $z~\cancel R~x$ and $z~\cancel{R'} ~x$.

The join $R \vee R'$ is the transfer system generated by $R$ and $R'$, and \autoref{prop:gen} implies that the only new relations in $\langle R, R' \rangle$ are those added via closure under transitivity. Thus there exists some $w\in P$ such that $z~R~w~R'~x$ or $z~R'~w~R~x$. Without loss of generality assume $z~R~w~R'~x$.

We claim $z=\text{min }w^{\downarrow}_{R}$. Indeed if we had $z'~R~w$ where $z'<z$ then $z'~(R\vee R')~x$ would contradict the minimality of $z$. Hence $z= \text{min }w^{\downarrow}_{R}$. Since $\chi^R=\chi^{R'}$ we then have $z= \text{min }w^{\downarrow}_{R'}$ which gives that $z~R'~w$. Thus $z~R'~x$ contradicting the minimality of $y$. So we get that $\text{min }x^{\downarrow}_{R\vee R'}=\text{min }x^{\downarrow}_{R}=\text{min }x^{\downarrow}_{R'}$ for all $x\in P$, as required.
\end{proof}

It turns out that the maxima of the fibers of $\chi$ are exactly the saturated transfer systems of \autoref{defn:sat}:

\begin{theorem}\label{thm:maxsat}
Each fiber of $\chi\colon \Tr(P)\to \End^\circ(P)$ has a greatest element; furthermore, the set of greatest elements of fibers of $\chi$ is exactly the set of saturated transfer systems on $P$.
\end{theorem}

\begin{proof}
Since $P$ is a finite lattice, so is $\Tr(P)$. Since each fiber of $\chi$ is finite, \autoref{lem:chijoins} shows that each fiber contains a greatest element. If $R\in \Tr(P)$, let $R_{\text{sat}}$ be the minimal saturated transfer system containing $R$. Then for $R\in \Tr(P)$, it suffices to show that $\operatorname{max}\{R'\in \Tr(P):\chi^{R'}=\chi^{R}\}=R_{\text{sat}}$. We first show that $\chi^{R_{\text{sat}}}=\chi^{R}$. Notice that the only new condition for saturated transfer systems is that if $x~R~y\leq z$ and $x~R~z$, then $y~R~z$. Because $x\leq y$, adding that $y~R~z$ leaves $\text{min }x^{\downarrow}_{R_{\text{sat}}}$ the same. Thus we see that $\chi^{R}=\chi^{R_{\text{sat}}}$.

It remains to show that $R_{\text{sat}}$ is maximal among those $R'\in \Tr(P)$ such that $\chi^{R'}=\chi^{R}$. Suppose that $R'>R_{\text{sat}}$ with $\chi^{R'}=\chi^{R}$. Then let $x~R'~y$ be such that $(x,y)\not\in R,R_{\text{sat}}$. This is equivalent to saying that for all $q~R~x$, $q\not\mathrel{R} y$. Let $z=\chi^{R}(y)=\chi^{R'}(y)$. The reader may check from definitions that $x$ and $z$ are incomparable, whence $x\wedge z< x,z$. By restriction closure, we have $x\wedge z~R'~z$ which by transitivity tells us that $z\wedge x ~R'~ y$. And in fact because $z\wedge x < x,z$ we get that this violates the minimality of $z$. Thus in every case there is a contradiction, so no such $R'$ exists. Hence we get that $\text{max }\{R'\in \Tr(P):\chi^{R'}=\chi^{R}\}=R_{\text{sat}}$.
\end{proof}

\begin{remark}
Restricting $\chi$ to saturated transfer systems  provides an antitone bijection between saturated transfer systems on $P$ and interior operators on $P$. The fact that these structures are in bijection was first observed in unpublished work of Balchin--MacBrough--Ormsby \cite{closure}. Their method is quite different. Indeed, they produce a monotone Galois connection $(F,G)$ with
\[
\begin{aligned}
    F\colon 2^P&\longrightarrow \Tr(P)\\
    S&\longmapsto \langle x\le \top\mid x\in S\rangle
\end{aligned}
\]
and
\[
\begin{aligned}
    G\colon \Tr(P)&\longrightarrow 2^P\\
    R&\longmapsto \{x\in P\mid x~R~\top\}.
\end{aligned}
\]
This restricts to a Galois correspondence between images. It turns out that $F(2^P)$ consists of \emph{cosaturated} transfer systems (those generated by relations involving $\top$) and $G(\Tr(P))$ consists of \emph{Moore families} --- the collections of `closed sets' for \emph{closure operators} on $P$ (monotone endomorphisms that are idempotent and extensive). It is further the case that cosaturated transfer systems on a finite lattice $P$ are in bijection with saturated transfer systems on $P^{\mathrm{op}}$, while closure operators on $P$ are in bijection with interior operators on $P^{\mathrm{op}}$. Replacing $P$ with $P^{\mathrm{op}}$ supplies a bijection between interior operators on $P$ and saturated transfer systems on $P$ which ultimately produces the same bijection found here.
\end{remark}

\begin{remark}\label{rmk:closure}
Interior operators on a finite lattice $P$ are in natural bijection with \emph{interior systems} (submonoids of $(P,\vee)$) and comonads on $P$ (viewed as a category). If $P$ is self-dual (\emph{e.g.}, if $P$ is the subgroup lattice of an Abelian group), then interior operators are further in bijection with closure operators, submonoids of $(P,\wedge)$, and monads on $P$.

In case $P$ is the finite (self-dual) Boolean lattice $[1]^n$, all of these structures are counted by OEIS entry A102896, with values
\[
  1,\quad 2,\quad 7,\quad 61,\quad 2\ 480,\quad 1\ 385\ 552,\quad 75\ 973\ 751\ 474,\quad 14\ 087\ 648\ 235\ 707\ 352\ 472,
\]
for $n=0,1,\ldots,7$. No additional values of this sequence are known, but its base-2 logarithm is known to grow at the rate of $\binom{n}{\lfloor n/2\rfloor}$ (see \cite{Kleitman}). This provides a lower bound on the asymptotic growth of $|\Tr([1]^n)|$. We leave it to future research to study the relative sizes of $|\Tr(P)|$ and $|\End^\circ(P)|$ for classes of finite lattices $P$.
\end{remark}

\autoref{thm:maxsat} gives an excellent way of classifying and identifying maximal elements in the fibers. It can also be shown that there will be minimal elements for these fibers, but the description of these is less nice than simply taking the saturated hull.

\begin{lemma}\label{lem:chimeets}
    Let $P$ be a lattice and let $R,R'\in \Tr(P)$, then $\chi^{R}=\chi^{R'}\implies \chi^{R\wedge R'}=\chi^{R}=\chi^{R'}$.
\end{lemma}
\begin{proof}
This is a straightforward consequence of $R\wedge R'$ being the intersection of (relations in) $R$ and $R'$.
\end{proof}

\begin{theorem}\label{thm:minchi}
Each fiber of $\chi\colon \Tr(P)\to \End^\circ(P)$ has a least element; namely for $R\in \Tr(P)$,
\[
  \min \chi^{-1}\{\chi^R\} = \tilde{R}=\langle (x,y)\mid\chi^{R}(y)=x \text{ or }x=y\rangle.
\]
\end{theorem}

\begin{proof}
By \autoref{lem:chimeets} and the fact that $P$ is a finite lattice, we know that such a least element exists. By the same argument given in the proof of \autoref{thm:interior}, we have that $\chi^{\tilde{R}}=\chi^{R}$. To show that $\tilde{R}$ is minimal, suppose for contradiction that there exists some $R'\in\Tr(P)$ such that $R'<\tilde{R}$ and $\chi^{R'}=\chi^{R}$. Then there must exist $x,y\in P$ such that $\chi^{R}(y)=x$ but $x\not\mathrel{\tilde{R}}y$. Thus we see that $\chi^{R'}(y)\ne x=\chi^{R}(y)$, a contradiction.
\end{proof}

\begin{remark}
In their totality, the results of this section demonstrate how $\chi$ gives instructions for decomposing $\Tr(P)$ into a collection of disjoint intervals. Associated with each interior operator $f\in \End^\circ(P)$ is the fiber $\chi^{-1}\{f\}$ which is of the form $[R,R'] = \{R''\in \Tr(P)\mid R\le R''\le R'\}$ where $R$ is a transfer system of the type given in \autoref{thm:minchi}, and $R'$ is the saturated hull of $R$.
\end{remark}

\section{Modular lattices and matchstick games}\label{sec:matchstick}

The ``matchstick game'' was introduced by Hafeez--Marcus--Ormsby--Osorno in \cite{hmoo} to enumerate the saturated transfer systems on rectangular lattices. Their results hinge on showing saturated transfer systems are in bijection with special subsets of covering relations called ``saturated covers''. We generalize the matchstick game to modular lattices. 

\begin{defn}
A \emph{modular lattice} is a lattice $(M, \leq )$ where all $a, b, x\in M$ satisfy the \emph{modular law}:
\[
a \leq b \implies a \vee (x \wedge b) = (a \vee x) \wedge b.
\]
\end{defn}

While subgroup lattices are \emph{not} modular in general, it is well-known that the lattice of normal subgroups of a group is always modular; in particular, if $G$ is Abelian (or Dedekind), then $\Sub(G)$ is modular.

\begin{defn}\label{defn:satcov}
Let  $(M, \leq )$ be a modular lattice and let $Q\subseteq {\le}$ be a subset of covering relations for $M$. We call $Q$ a \emph{saturated cover} when the following conditions hold:
\begin{enumerate}
    \item[(1)] For $x, y\in M$, if $x ~Q ~(x \vee y)$ then $(x \wedge y) ~Q~ y$.
    \item[(2)] Suppose that $x$ and $y$ are covered by $x\vee y$ and cover $x\wedge y$ (we call such a tetrad a \emph{covering diamond}); if three of the four covering relations between $x, y, x \wedge y$, and $x \vee y$ are in $Q$, then the fourth covering relation is in $Q$ as well.
\end{enumerate} 
\end{defn}

\begin{example}
The diagram in \autoref{fig:satcovers} depicts all of the saturated covers for the lattice $[2]^{*3}\cong \Sub(C_2\times C_2)$. (Here $[2] = \{0<1<2\}$ and the fusion operation $*$ is defined in \autoref{defn:fusion}.) We will return to this example when discussing \autoref{thm:fusion}. The reader may also jump to \autoref{ex:cube} to see the saturated covers for the cube $[1]^3$.

\begin{figure}[ht]
\input{saturated_covers}
\caption{Saturated covers for the lattice $[2]^{*3}\cong \Sub(C_2\times C_2)$.}
\label{fig:satcovers}
\end{figure}
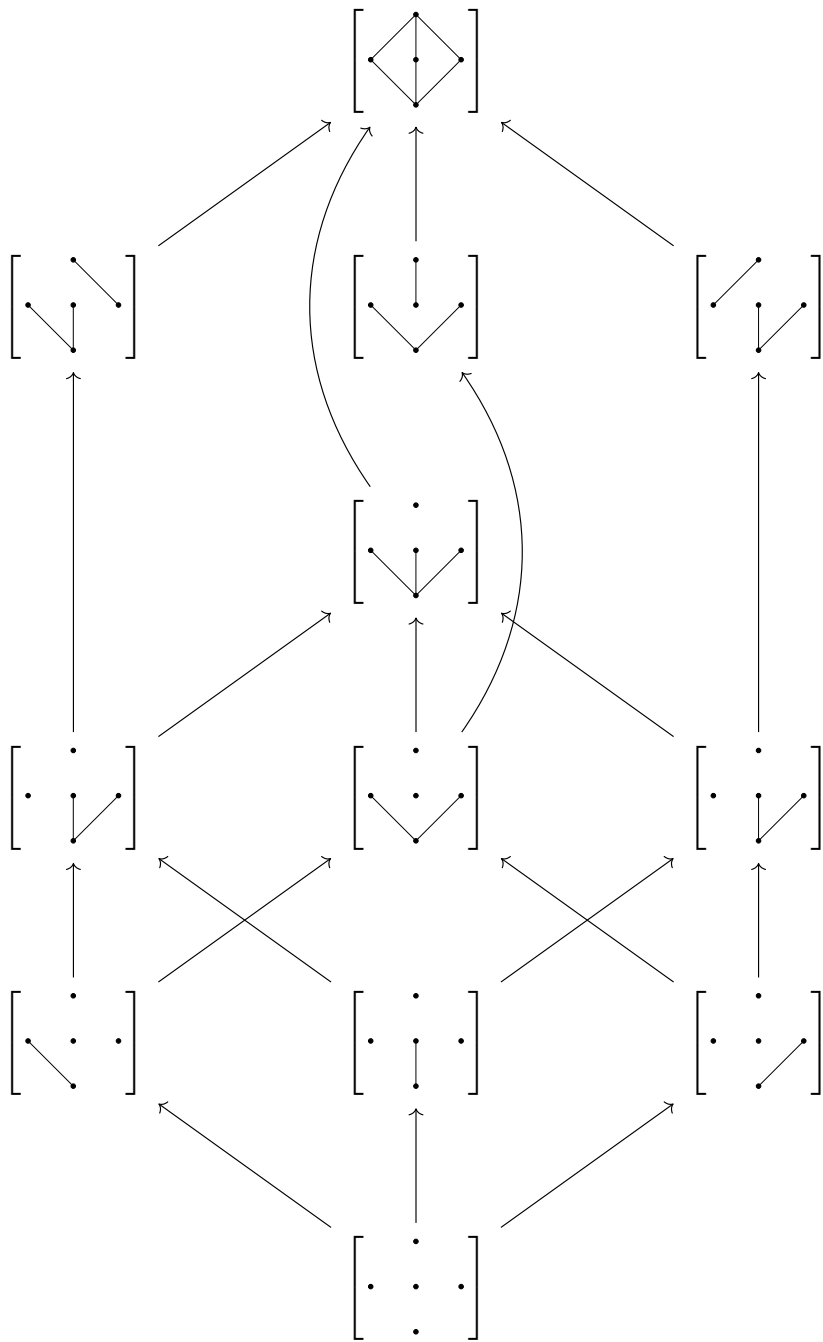
\end{example}

We now show that the set of covering relations for any saturated transfer system $R$ on a finite modular lattice $(M,\leq)$ forms a saturated cover. We will use the following lemma, which is standard and may be checked via an elementary argument.
\begin{lemma}\label{lemma:mod}
Let $(M,\leq)$ be a modular lattice. For any $x,y\in M$, if $x\vee y$ covers $x$, then $y$ covers $x\wedge y$.\hfill\qedsymbol
\end{lemma}

\begin{lemma}\label{lemma:3of4}
Let $R$ be a saturated transfer system on a modular lattice $M$ and let $R_{cov}$ denote the covering relations of $M$ that are in $R$. If three of the four edges in a covering diamond formed by $x, y, x \wedge y, x \vee y\in M$ exist in $R_{cov}$, then the fourth edge is in $R_{cov}$ as well.
\end{lemma}
\begin{proof}
If $x,y,x\vee y,$ and $x\wedge y$ form a covering diamond in $M$ with 3 out of 4 edges in $R_{cov}$, then either $(x\wedge y)~R_{cov}~x~R_{cov}~(x\vee y)$ or $(x\wedge y)~R_{cov}~y~R_{cov}(x\vee y)$. In the first case, the restriction rule implies that $(x\wedge y)~R_{cov}~y$, and in the second case we get $(x\wedge y)~R_{cov}~x$. In either case, transitivity implies $(x\wedge y)~R~(x\vee y)$, and then the 2-out-of-3 rule for saturated transfer systems gives us the final edge of the covering diamond.
\end{proof}

\begin{prop}\label{prop:sys_to_cov}
If $R$ is a saturated transfer system on a modular lattice $M$ and $R_{cov}$ is the collection of covering relations (of $M$) in $R$, then $R_{cov}$ is a saturated cover.
\end{prop}

\begin{proof}
Condition (1) of \autoref{defn:satcov} is an immediate consequence of the restriction rule and \autoref{lemma:mod}. Condition (2) follows from \autoref{lemma:3of4}.
\end{proof}

We now prove that any saturated cover $Q$ on a finite modular lattice $(M,\leq)$ generates a saturated transfer system $\langle Q\rangle$ on $M$. Our method relies on the existence of a \emph{grading} on $M$ in the sense of the following definition.

\begin{defn}
A \emph{graded poset} is a poset $(P,\leq)$ equipped with a rank function $\rho\colon P\to\mathbb{N}$ such that:
\begin{enumerate}
    \item[(1)] if $x<y$, then $\rho(x)<\rho(y)$, and
    \item[(2)] if $y$ covers $x$, then $\rho(y)=\rho(x)+1$
\end{enumerate}
\end{defn}

The following result is standard, and follows from \emph{Dedekind's modularity criterion}: modular lattices are precisely those lattices with no pentagonal sublattices. (See, for instance, \cite[Theorem 4.4]{blyth}.)

\begin{prop}
If $(M,\leq)$ is a finite modular lattice, then $M$ is a graded poset with rank function $\rho\colon M\to\mathbb{N}$ that maps each element $x$ to the number of covering relations required to travel from the minimal element of $M$ to $x$.\hfill\qedsymbol
\end{prop}

Given a grading on a poset, we get a natural notion of distance between elements:

\begin{defn}
If $(P,\le)$ is a graded poset with rank function $\rho$ and $x\le z$ are elements of $P$, we say the \emph{length of the interval $[x,z]$}, or the \emph{distance from $x$ to $z$}, is
\[
    \ell[x,z] := \rho(z)-\rho(x).
\]
\end{defn}

The length function on a finite modular lattice gives us a way to perform induction when proving the following crucial lemma.

\begin{lemma}\label{lemma:sat}
Suppose $(M,\leq)$ is a finite modular lattice, $Q$ is a saturated cover on $M$, $x\le y\le w\le z\in M$, and $x~\langle Q\rangle~z$. Then $y~\langle Q\rangle~ w$.
\end{lemma}
\begin{proof}
Fix $x\in M$. We perform induction on $\ell[x,z]$. The base case is true since $\ell[x,z] = 0$ implies $x=z$, and $\langle Q\rangle$ is reflexive.

For the induction hypothesis, fix $n\ge 0$ and assume the claim holds for any $z\ge w\ge y\ge x\in M$ such that $x~\langle Q\rangle~ z$ and $\ell[x,z]\leq n$. Now, say we have $z\ge x\in M$ such that $x~\langle Q\rangle~ z$ and $\ell[x,z]=n+1$. Since $x~\langle Q\rangle~ z$, there exists a path of covering relations
\[
  x=z_0~Q~z_1~Q~z_2~Q~\cdots~Q~z_{n+1}=z
\]
such that $\ell[x,z_i]=i$ for all $i$. For any $e\in[x,z]$ such that $\ell[x,e]=n$, we claim that $e~\langle Q\rangle~ z$. To see why this is true, note that if $e=z_n$ then $z_n~Q~z_{n+1}$ ensures $e~\langle Q\rangle~ z$. If $e\neq z_n$, then by the restriction rule, $x~\langle Q\rangle~ z$ and $e\leq z$ implies $(x\wedge e)~\langle Q\rangle~ e$. Since $x\leq e$, this means $x~\langle Q\rangle~ e$. We know $e$ and $z_n$ are incomparable because they are both distance $n$ from $x$. We also know $x\leq(z_n\wedge e)$ by the definition of a meet. Now, we have $z_n\wedge e,z_n\in[x,z_n]$ where $x~\langle Q\rangle~z_n$, $(z_n\wedge e)\leq z_n$, and $\ell[x,z_n]=n$, so by the induction hypothesis, $(z_n\wedge e)~\langle Q\rangle~ z_n$. The same argument can be used to show that $(z_n\wedge e)\leq e\in[x,e]$ implies $(z_n\wedge e)~\langle Q\rangle~ e$. Since $z_n,e,z_n\wedge e$, and $z_n\vee e=z$ form a covering diamond in $Q$ with three of four edges, it follows that the fourth edge is in $Q$, \emph{i.e.}, $e~Q~z$. This proves that all elements in $[x,z]$ that are distance $n$ from $x$ are $\langle Q\rangle$ related to $z$. 

Now fix $y,w\in[x,z]$ such that $y\leq w$. We need to prove that $y~\langle Q\rangle~ w$. We first consider the case where
\[
  \ell[x,y]\leq n\qquad\text{and}\qquad \ell[x,w]=n+1.
\]
Then $w=z$ and we need to prove $y~\langle Q\rangle~ z$. If $\ell[x,y]=n$, our argument from the previous paragraph ensures that $y~\langle Q\rangle~ z$, so we may assume $\ell[x,y]=i$ where $i<n$. We claim there exists $y'\in [x,z]$ such that $\ell[x,y']=n$ and $y\leq y'$. To see this, we will build path of covering relations
\[
  y=y_0<y_1<y_2<\cdots <y_{n-i}=y'
\]
in $[x,z]$ to connect $y$ and $y'$. Again let
\[
  x=z_0~Q~z_1~Q~z_2~Q~\cdots~Q~z_{n+1}=z
\]
be a chain of covering relations in $Q$. Our rule for constructing the path of $y_j$'s is as follows: 
\begin{enumerate}
    \item[(i)] Start by setting $j=0$.
    \item[(ii)] If $y_j=z_{i+j}$, then set $y_k=z_{i+k}$ for all $k>j$ and we are done; if not, then move on to step (iii).
    \item[(iii)] Given that $y_j\neq z_{i+j}$, we know that $y_j$ is incomparable to $z_{i+j}$ as they have the same rank. Set $y_{j+1}=y_j\vee z_{i+j}$, increment the value of $j$ by one, and go back to step (ii).
\end{enumerate}
We repeat the process above until we either satisfy step (ii), ending the process, or we increment $j$ enough times to reach $j=n-i$, at which point we stop the process and get a sequence where $y_{j+1}=y_j\vee z_{i+j}$ for all $j<n-i$. In either case, the resulting sequence connects $y$ to an element $y'\in[x,z]$ with distance $n$ from $x$. Since the restriction rule forces $x~\langle Q\rangle~ y'$, the fact that $y\leq y'\in[x,y']$ implies that $y~\langle Q\rangle~ y'$ by the inductive hypothesis. But of course, we know that $y'~\langle Q\rangle~ z$ from earlier observations regarding elements with distance $n$ from $x$, so by transitivity we have $y~\langle Q\rangle~ z$.

We now consider the case where both $\ell[x,y]$ and $\ell[x,w]$ are at most $n$. 
By restriction, $x~\langle Q\rangle~ w$. Thus, by the inductive hypothesis, $y\leq w\in[x,w]$ implies $y~\langle Q\rangle~ w$.

Finally, if $\ell[x,y]=\ell[x,w]=n+1$, then $y=w=z$ and this case is handled by the reflexivity of $\langle Q\rangle$. This completes the induction step and the proof.
\end{proof}

With this lemma in hand, we are now ready to show that saturated covers generate saturated transfer systems.

\begin{prop}\label{prop:cov_to_sys}
If $Q$ is a saturated cover on a finite modular lattice $(M,\leq)$, then $\langle Q\rangle$ is a saturated transfer system.
\end{prop}
\begin{proof}
By \autoref{defn:gen}, $\langle Q\rangle$ is a transfer system. \autoref{lemma:sat} proves that $\langle Q\rangle$ is saturated because given $x<y<z\in M$, the only case of the 2-out-of-3 rule that is not handled by restriction or transitivity is the case where $x~\langle Q\rangle~ y$ and $x~\langle Q\rangle~ z$. In this case, we have $y\leq z\in[x,z]$, so \autoref{lemma:sat} says that $y~\langle Q\rangle~ z$.
\end{proof}

We need one more lemma before proving the main theorem of this section.

\begin{lemma}\label{lemma:Qres}
If $Q$ is a saturated cover on a finite modular lattice $(M,\le)$, then $Q$ is closed under restriction (ignoring identity relations forced by the restriction rule).
\end{lemma}
\begin{proof}
Assume to the contrary that there are elements $b,\ell,r,t\in M$ such that $r~Q~t$, $\ell\leq t$, $b=\ell\wedge r\neq\ell$, but $b~{\cancel Q}~\ell$. Since $t$ covers $r$, we know $r<t$ such that there is no element $w\in M$ with $r<w<t$. Then, given any $e\in M$ such that $e\geq r$ and $e\geq\ell$, if $e=r$ then $\ell\leq r$ and $\ell\wedge r=\ell$, a contradiction. If $e\neq r$, then $e\geq t$. This means $t=\ell\vee r$, so $r~Q~t$ forces $b~Q~\ell$ by condition (1) of \autoref{defn:satcov}, a contradiction.
\end{proof}

We can now prove that saturated transfer systems on finite modular lattices are in bijection with saturated covers. It is this theorem which we view as a ``matchstick game'' whose solutions (saturated covers) enumerate saturated transfer systems on finite modular lattices. (We refer to this as a ``matchstick game'' because one can envision trying to lay down matchsticks  --- \emph{i.e.}, covering relations --- on the lattice in a way that satisfies the rules of \autoref{defn:satcov}.)

\begin{theorem}\label{thm:matchstick}
Let $Q$ be a subset of covering relations on a finite modular lattice $(M, \leq )$. Then, $Q$ is the set of covering relations within a saturated transfer system if and only if $Q$ is a saturated cover, and this correspondence provides a bijection between saturated covers and saturated transfer systems on $M$.
\end{theorem}
\begin{proof}
Given a saturated cover $Q$, \autoref{lemma:Qres} tells us that $\langle Q\rangle$ is generated from $Q$ by first enforcing reflexivity on $Q$, and then enforcing transitivity on $Q$. Hence, $Q$ is indeed the set of covering relations of $\langle Q\rangle$. Then, \autoref{prop:cov_to_sys} states that $\langle Q\rangle$ is a saturated transfer system. 

Given a saturated transfer system with covering relations $Q$, \autoref{prop:sys_to_cov} tells us that $Q$ is a saturated cover.

Finally, by \autoref{prop:sat_det_cov}, saturated transfer systems are generated by their covering relations, so the assignments $Q\mapsto \langle Q\rangle$ and $R\mapsto R_{cov}$ are mutually inverse.
\end{proof}

\begin{example}\label{ex:cube}
\autoref{tab:satcov} lists all of the saturated covers for $[1]^3\cong \Sub(C_p\times C_q\times C_r)$ for $p,q,r$ distinct primes. Given a saturated cover $R$ on $[1]^3$, let $R_t$ denote the restriction of $R$ to the ``top'' face (with final coordinate $1$), and let $R_b$ denote the restriction of $R$ to the ``bottom'' face (with final coordinate $1$), both thought of as a saturated covers of $[1]^2$. The restriction rule for saturated covers implies that $R_t\le R_b$. This allows us to organize our enumeration according to the bottom face, and these are listed in the first column of the table. Each row of the table corresponds to a different choice of $R_t$ refining the bottom face. In the first row of the table we see all possible legal arrangements of vertical covering relations, and the subsequent rows only include those choices of verticals that are compatible with the given $R_t\le R_b$. Note that the total number of saturated covers for $[1]^3$ is $61$, which is indeed the number of interior operators on $[1]^3$ (see \autoref{rmk:closure}).
\end{example}

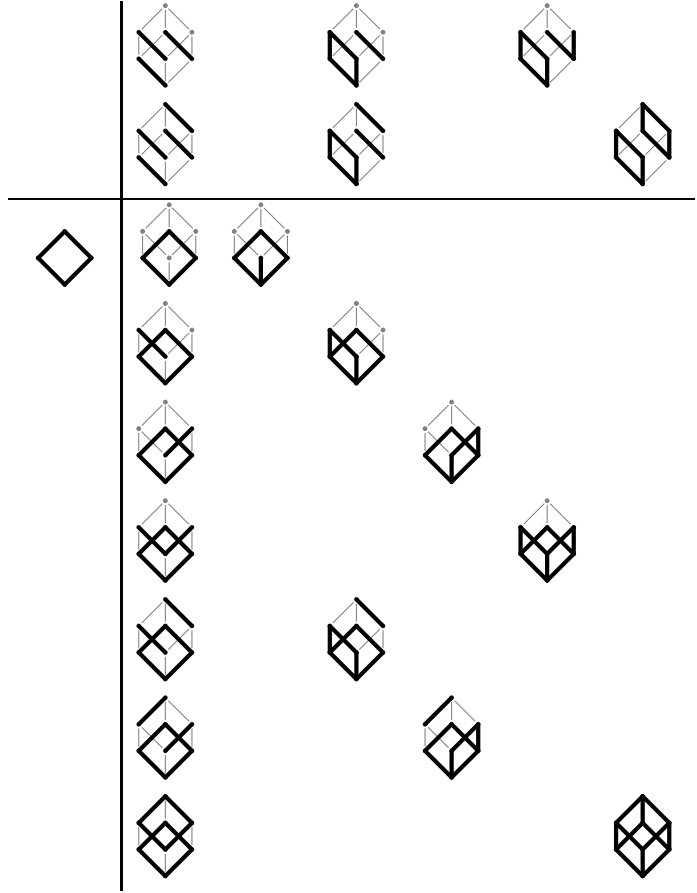
\input{saturated_diagrams}

\begin{remark}
Our method of record-keeping in the previous example also suggests an enumeration of saturated transfer systems on lattices of the form $P\times [1]$ in terms of certain structures on $P$. Indeed, each saturated cover $R$ for $P\times [1]$ restricts to an interval $R_t\le R_b$ of saturated covers for $P$. Given such an interval, one must enumerate the collections of ``vertical'' covering relations that are closed under restriction and still satisfy 3-out-of-4. These vertical relations can be put in bijection with antichains in $P$ (subsets of mutually incomparable elements) satisfying certain rules. We have not yet been able to turn this into an effective tool for enumeration, so we leave further details to the reader.
\end{remark}

\section{Categorical properties of transfer systems}\label{sec:cat}

We now turn to some categorical properties of transfer systems, leading to a new recursion (\autoref{thm:fusion}) for transfer systems on a particular pushout of two lattices which we call the \emph{fusion} (see \autoref{defn:fusion}). For completeness, we begin this discussion by considering the status of $\Tr$ as a functor.

The assignment sending a bounded lattice $P$ to its lattice of transfer systems $\Tr(P)$ turns out to be a functor, though one needs to be careful with morphisms. Given a monotone map of lattices $f\colon P \to Q$, we can form $\Tr(f)\colon \Tr(P)\to \Tr(Q)$ by setting
\[
  [\Tr(f)](R)=\langle (f(x),f(y)) \mid x,y\in P, x~R~y\rangle
\]
where $R\in \Tr(P)$ and $\langle ~\rangle$ is as in \autoref{defn:gen}. It is clear $\Tr(id_P)=id_{\Tr(P)}$, but it is unclear if $\Tr(f \circ g)=\Tr(f)\circ \Tr(g)$ for composable maps $f$ and $g$. In fact, it turns out that if we only require monotone maps, then we could have extra relations in $\Tr(f)\circ \Tr(g)(R)$ that may not be present in $\Tr(f \circ g)(R)$. We illustrate this in \autoref{ex:composablemaps}.

\begin{example}\label{ex:composablemaps}
    Consider the following composition of monotone maps of lattices:
    
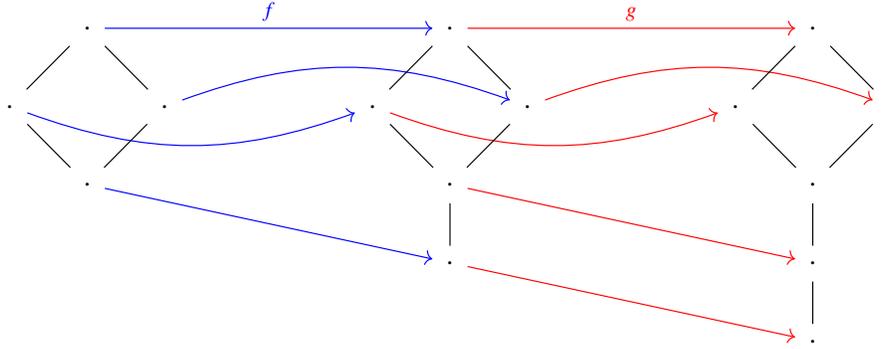
\begin{center}
    \begin{tikzcd}[row sep=1.5em, column sep=1.5em]
    &
    \cdot \arrow[rrrrrr,blue,"f"{blue}]&
    &
    &
    &
    &
    &
    \cdot \arrow[rrrrrr,red,"g"{red}]&
    &
    &
    &
    &
    &
    \cdot \\
    \cdot \arrow[ur,dash] \arrow[rrrrrr,blue,bend right=20]&
    & 
    \cdot \arrow[ul,dash] \arrow[rrrrrr,blue,bend left=20]&
    &
    &
    &
    \cdot \arrow[ur,dash]\arrow[rrrrrr,red,bend right=20]&
    &
    \cdot \arrow[ul,dash] \arrow[rrrrrr,red,bend left=20]&
    &
    &
    &
    \cdot \arrow[ur,dash]&
    & 
    \cdot \arrow[ul,dash]\\
    &
    \cdot \arrow[ul,dash] \arrow[ur,dash] \arrow[blue,drrrrrr]&
    &
    &
    &
    &
    &
    \cdot \arrow[ul,dash] \arrow[ur,dash] \arrow[red,drrrrrr]&
    &
    &
    &
    &
    &
    \cdot \arrow[ul,dash] \arrow[ur,dash]\\
    &
    &
    &
    &
    &
    &
    &
    \cdot \arrow[u,dash] \arrow[red,drrrrrr]&
    &
    &
    &
    &
    &
    \cdot \arrow[u,dash]\\
    &
    &
    &
    &
    &
    &
    &
    &
    &
    &
    &
    &
    &
    \cdot \arrow[u,dash]
    \end{tikzcd}
    \end{center}
\end{example}
Consider the transfer system $R$ on $[1]\times [1]$ shown below. One can verify $\Tr(g\circ f)(R)$ and $\Tr(g)\circ \Tr(f)(R)$ are given by the illustrated transfer systems. In particular, note the extra relations added in $\Tr(g) \circ \Tr(f)(R)$ due to the intermediate closure operation.
\begin{center}
\begin{multicols}{3}
\begin{center}
\underline{$R\in \Tr([1]\times[1])$}
\end{center}
\columnbreak
\begin{center}
\underline{$\Tr(g \circ f)(R)$}
\end{center}
\columnbreak
\begin{center}
\underline{$\Tr(g)\circ \Tr(f)(R)$}
\end{center}
\end{multicols}
\end{center}
\begin{center}
\begin{multicols}{3}
\begin{center}
\begin{tikzcd}[row sep=1.5em, column sep=1.5em]
    &
    \cdot \\
    \cdot &
    &
    \cdot \arrow[ul,dash]\\
    &
    \cdot \arrow[ul,dash]
\end{tikzcd}
\end{center}
\columnbreak
\begin{center}
\begin{tikzcd}[row sep=1.5em, column sep=1.5em]
    &
    \cdot \\
    \cdot &
    &
    \cdot \arrow[ul,dash]\\
    &
    \cdot \arrow[ul,dash]\\
    &
    \cdot\\
    &
    \cdot \arrow[uuul,dash] \arrow[uu,dash,bend right=30] \arrow[u,dash]
\end{tikzcd}
\end{center}
\columnbreak
\begin{center}
\begin{tikzcd}[row sep=1.5em, column sep=1.5em]
    &
    \cdot \\
    \cdot &
    &
    \cdot \arrow[ul,dash]\\
    &
    \cdot \arrow[ul,dash]\\
    &
    \cdot \arrow[u,dash] \arrow[uul,dash]\\
    &
    \cdot \arrow[uuul,dash] \arrow[uu,dash,bend right=30] \arrow[u,dash]
\end{tikzcd}
\end{center}
\end{multicols}
\end{center}

This example shows that $\Tr(f\circ g)\neq \Tr(f) \circ \Tr(g)$ for general monotone maps $f$, $g$. It turns out that the phenomenon observed above is essentially the only obstruction to preservation of composition. To see the general case, let $P,Q,L$ be lattices and $f \colon P \to Q$, $g \colon Q \to L$ be monotone maps. Also let $R \in \Tr(P)$. We will use $fR$ to denote the relation $\{(f(x),f(y)) \mid x,y\in P, x~R~y\} \cup \{(z,z): z\in Q\}$ (not taking restrictive or transitive closure) and $f_{*}R$ to denote the relation $\langle fR \rangle=\Tr(f)(R)$. Consider a restriction diagram in $P$ of the form
\begin{center}
\begin{tikzcd}[ampersand replacement=\&]
    x \arrow[r] \&
    z \\
    x\wedge y \arrow[u,"R"] \arrow[r]\&
    y. \arrow[u, "R"]
\end{tikzcd}
\end{center}

Then we can step by step find all relations in $L$ that are added by either $\Tr(g \circ f)$ or $\Tr(g)\circ \Tr(f)$ and compare. Doing this, we obtain the following diagrams:

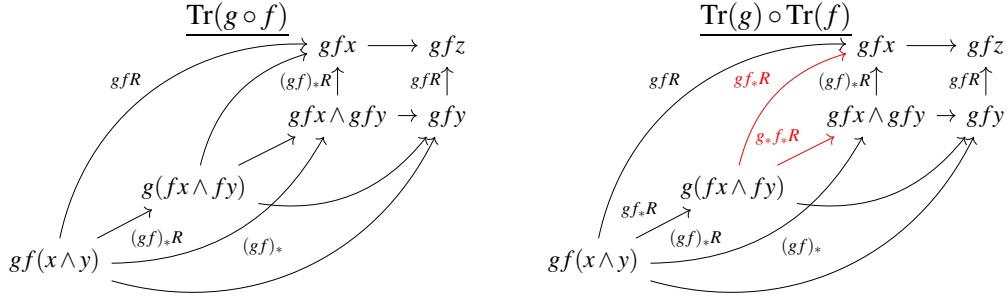
\begin{center}
\begin{multicols}{2}
\begin{center}
\underline{$\Tr(g \circ f)$}
\end{center}
\columnbreak
\begin{center}
\underline{$\Tr(g)\circ \Tr(f)$}
\end{center}
\end{multicols}
\vspace{0.4cm}
\begin{multicols}{2}
    \begin{center}
    \begin{tikzcd}[ampersand replacement=\&,row sep=1.2em, column sep=0.8em,transform canvas={scale=0.85}]
    \&
    \&
    gfx \arrow[r] \&
    gfz \\
    \&
    \&
    gfx\wedge gfy \arrow[u,"(gf)_{*}R"] \arrow[r]\&
    gfy \arrow[u, "gfR"]\\
    \&
    g(fx\wedge fy) \arrow[uur,bend left=30] \arrow[urr, bend right=30] \arrow[ur] \\
    gf(x\wedge y) \arrow[uuurr,bend left=40, "gfR"] \arrow[uurrr, bend right=40] \arrow[ur, "(gf)_{*}R"swap] \arrow[uurr, bend right=30, "(gf)_{*}"swap]
    \end{tikzcd}
    \end{center}
    \columnbreak
    \begin{center}
    \begin{tikzcd}[ampersand replacement=\&,row sep=1.2em, column sep=0.8em,transform canvas={scale=0.85}]
    \&
    \&
    gfx \arrow[r] \&
    gfz \\
    \&
    \&
    gfx\wedge gfy \arrow[u,"(gf)_{*}R"] \arrow[r]\&
    gfy \arrow[u, "gfR"]\\
    \&
    g(fx\wedge fy) \arrow[uur,red, bend left=30,"gf_{*}R" red] \arrow[urr, bend right=30] \arrow[ur, "g_{*}f_{*}R" red, red] \\
    gf(x\wedge y) \arrow[uuurr,bend left=40, "gfR"] \arrow[uurrr, bend right=40] \arrow[ur, "gf_{*}R", "(gf)_{*}R"swap] \arrow[uurr, bend right=30, "(gf)_{*}"swap]
\end{tikzcd}   
\end{center}
\end{multicols}
\end{center}
\vspace{2cm}

So if we compare, in the $\Tr(g)\circ\Tr(f)$ case, we get two extra maps due to the intermediate closure; 
$g(fx\wedge fy) \xrightarrow{g_{*f_{*R}}}gfx\wedge gfy$ and $g(fx \wedge fy) \xrightarrow{gf_{*}R} gfx$ as seen in red. Requiring our monotone maps to preserve meets is sufficient to get rid of this issue because everything collapses to the typical restriction square in $L$ of the form

\begin{center}
\begin{tikzcd}[ampersand replacement=\&]
    gfx \arrow[r] \&
    gfz \\
    gfx\wedge gfy \arrow[u,"(gf)_{*}R"] \arrow[r]\&
    gfy. \arrow[u, "gfR"]
\end{tikzcd}
\end{center}

It turns out that requiring our maps to preserve meets is enough to guarantee functorality. (The reader should note the maps $f$ and $g$ fail to preserve meets in \autoref{ex:composablemaps}.) Let $\mathbf{Lat}_m$ and $\mathbf{Lat}_{\wedge}$ be the categories of (bounded) lattices with monotone maps and meet preserving monotone maps, respectively. We have the following functorality statement.

\begin{theorem}\label{thm:functor}
The assignment $\Tr\colon \mathbf{Lat}_{\wedge}\to \mathbf{Lat}_{m}$ taking $P$ to $\Tr(P)$ and a meet-preserving monotone map $f\colon P\to Q$ to
\[
\begin{aligned}
    \Tr(f)\colon \Tr(P)&\longrightarrow \Tr(Q)\\
    R&\longmapsto \langle (f(x),f(y))\mid x,y\in P, x~R~y\rangle
\end{aligned}
\]
is a functor.
\end{theorem}
\begin{proof}
We leave the proof of this to the reader, noting that \autoref{prop:gen} is a useful ingredient.
\end{proof}

\begin{remark}
It is not the case that $\Tr$ is an endofunctor on $\mathbf{Lat}_{\wedge}$. To see why, consider the example $f\colon [2]\to [1]\times [1]$ given by $f(0)=(0,0)$, $f(1)=(1,0)$, $f(2)=(1,1)$. One can verify $f$ is meet-preserving, but $\Tr(f)$ is not.
\end{remark}

The fact that $\Tr$ is a functor grants us new categorical territory to explore. We record a quick result here, and we make a remark on products at the end of the section.

\begin{prop}
    The functor $\Tr\colon \mathbf{Lat}_{\wedge}\to \mathbf{Lat}_{m}$ is not essentially surjective, \emph{i.e.}, not every lattice is isomorphic to the lattice of transfer systems of some lattice.
\end{prop}

\begin{proof}
One can show there exists no lattice $P$ such that $\Tr(P)\cong [2]$. We leave the details to the reader.
\end{proof}

\begin{remark}
    It is not known if $\Tr$ is essentially injective, \emph{i.e.}, if there are nonisomorphic $P,Q$ such that $\Tr(P)\cong \Tr(Q)$.
\end{remark}

We now define a binary operation $*$ on two (bounded) lattices $P,Q$ that we call \emph{fusion}. This corresponds to simply gluing together the top and bottom elements in $P$ and $Q$ while leaving everything else unchanged.  In \autoref{sec:CpCp}, we will see that this operation arises naturally when considering a certain class of subgroup lattices.

\begin{defn}\label{defn:fusion}
Let $(P,\leq_{P}),(Q,\leq_{Q})$ be two bounded lattices with $\top,\perp$ as the top and bottom elements respectively. We assume that $P\cap Q=\{\top,\perp\}$. The \emph{fusion} of $P$ and $Q$ is the bounded lattice with underlying set $P*Q :=P\cup Q$ and order defined by ${\le_{P*Q}} := {\leq_{P}}\cup {\leq_{Q}}$.
\end{defn}
\begin{example}
The following diagram depicts the Hasse diagram of the fusion of two lattices in terms of the original Hasse diagrams.

    
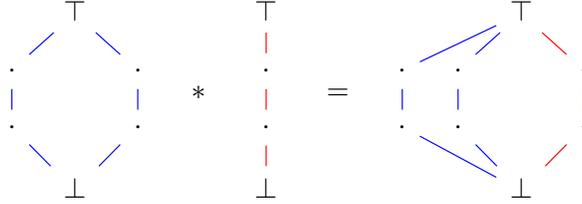
\begin{center}
    \begin{tikzcd}[row sep=0.75em, column sep=0.75em]
    &
    \top \\
    \cdot \arrow[ur,dash,blue] &
    &
    \cdot\arrow[ul,dash,blue] \\
    \cdot \arrow[u,dash,blue]&
    &
    \cdot \arrow[u,dash,blue]\\
    &
    \perp \arrow[ur,dash,blue] \arrow[ul,dash,blue]
    \end{tikzcd}
    \quad
    $*$
    \quad
    \begin{tikzcd}[row sep=0.75em, column sep=0.75em]
    \top \\
    \cdot \arrow[u,dash,red]\\
    \cdot \arrow[u,dash,red]\\
    \perp \arrow[u,dash,red]
    \end{tikzcd}
    \quad
    $=$
    \quad
    \begin{tikzcd}[row sep=0.75em, column sep=0.75em]
    &
    &
    \top \\
    \cdot \arrow[urr,dash,blue] &
    \cdot\arrow[ur,dash,blue] &
    &
    \cdot \arrow[ul,red,dash]\\
    \cdot \arrow[u,dash,blue]&
    \cdot \arrow[u,dash,blue]&
    &
    \cdot \arrow[u,dash,red]\\
    &
    &
    \perp \arrow[ul,dash,blue] \arrow[ull,dash,blue] \arrow[ur,dash,red]
    \end{tikzcd}
    \end{center}
\end{example}
We record two categorical facts about fusion here, leaving the details to the interested reader.

\begin{prop}
We have the following categorical properties of fusion:
\begin{enumerate}[(a)] 
    \item The fusion $P*Q$ is the pushout of the diagram $Q\xleftarrow{\top\mapsfrom 1, \perp\mapsfrom 0}[1]\xrightarrow{0\mapsto \perp, 1\mapsto \top}P$ in $\mathbf{Lat}_{m}$.
    \item $(\mathbf{Lat}_{m},*,[1])$ is symmetric monoidal.\hfill \qedsymbol
\end{enumerate} 
\end{prop}

We caution the reader that $P*Q$ is not the pushout of $Q\leftarrow [1]\to P$ in $\mathbf{Lat}_\wedge$. While we do have a diagram
    \begin{center}
    \begin{tikzcd}
        \left[1\right] \arrow[r,] \arrow[d] &
        P \arrow [d,"\iota_{P}"]\\
        Q \arrow[r,"\iota_{Q}"] &
        P*Q
    \end{tikzcd}
    \end{center}
in $\mathbf{Lat}_\wedge$ (where $\iota_{P},\iota_{Q}$ are the obvious embeddings), it is not the case that maps out of $P*Q$ given by the universal property in $\mathbf{Lat}_m$ will necessarily preserve meets.

\begin{remark}
One might expect the fusion to be the coproduct of $P$ and $Q$ in $\mathbf{Lat}_{m}$, but it is easy to construct examples to show this is not the case. However, this can be formulated as a coproduct if, in our category, we also restrict morphisms to preserve top and bottom elements, i.e. $\top \mapsto \top$ and $\perp \mapsto \perp$. In this case, $[1]$ would be initial, so we get a coproduct. 
\end{remark}

In order to determine the structure of $\Tr(P*Q)$ in terms of $\Tr(P)$ and $\Tr(Q)$, we will need the following lemma. In order to parse its statement, we need to extend the notion of transfer systems to posets that do not necessarily have meets. We make the convention here (as in \cite[Definition 2.3]{lifting}) that transfer systems on a poset $(P,\le)$ are partial orders $R$ on $P$ refining  $\le$  such that $x~R~y$ and $z\le y$ implies that for all maximal $w\le x,z$ we have $w~R~z$. Of course, when $x\wedge z$ exists in the above setup, it is the unique $w$ required to satisfy $w~R~z$, so this is compatible with the definition for lattices.

\begin{lemma}\label{lem:botelttsbij}
Transfer systems $R$ on a lattice $P$ such that ${\perp}~R~a$ for all $ a \in P$ are in bijection with transfer systems on $P \smallsetminus \{ \perp\}$.
\end{lemma}
\begin{proof}
Let the former set of transfer systems be denoted $T$. Given a transfer system $Q$ on $P\smallsetminus \{\perp\}$, let $\tilde Q$ denote the same relations on $P$ along with all relations from $\perp$. The reader may check that the assignments
\[
\begin{aligned}
    T &\longleftrightarrow \Tr(P\smallsetminus\{\perp\})\\
    R &\longmapsto R|_{P\smallsetminus\{\perp\}}\\
    \tilde Q &\longmapsfrom Q
\end{aligned}
\]
are mutually inverse.
\end{proof}

We now come to the main theorem of this section, which enumerates transfer systems on the fusion $P*Q$ of finite lattices $P,Q$ in terms of smaller posets. Recall that, given a transfer system $R$ on $P$, the \emph{minimal fibrant} element of $P$ (relative to $R$) is the (necessarily unique) least element $a\in P$ such that $a~R~\top$. In other words, the minimal fibrant is $\chi^R(\top)$. We write $\Tr_a(P)$ for the collection of transfer systems on $P$ with minimal fibrant $a$.

\begin{theorem}\label{thm:fusion}
Let $P,Q$ be finite lattices with fusion lattice $P*Q$. Then
\begin{multline*}
  |\Tr(P*Q)|
  =|\Tr(P\smallsetminus \{\top\})||\Tr(Q\smallsetminus \{\top\})| +|\Tr(P\smallsetminus \{\perp\})| |\Tr(Q\smallsetminus \{\perp\})|\\
  +\sum_{a\in P\smallsetminus \{\top,\perp\}}|\Tr_a(P)||\Tr(Q\smallsetminus \{\top,\perp\})| 
  +\sum_{b\in Q\smallsetminus \{\top,\perp\}}|\Tr(P\smallsetminus \{\top,\perp\})| |\Tr_b(Q)|.
\end{multline*}
\end{theorem}

\begin{proof}
We partition the enumeration by minimal fibrants. Suppose that $\top$ is the minimal fibrant for a transfer system on $P*Q$. Then as only $\top$ is related to itself, no relations in $P \smallsetminus \{\top\}$ can impose relations on $Q \smallsetminus \{\top\}$ by restriction. Thus we get all possible transfer systems by choosing transfer systems on $P \smallsetminus \{ \top \}$ and $Q \smallsetminus \{ \top \}$, accounting for the first term. Now suppose $\perp$ is the minimal fibrant for a transfer system on $P*Q$. By restriction, we then have that $\perp$ is related to all other elements. Then by \autoref{lem:botelttsbij}, the number of transfer systems is equal to $|Tr((P*Q) \smallsetminus \{\perp\})|$. As no element in $P \smallsetminus \{ \perp\}$ can be related to an element in $Q \smallsetminus \{\perp \}$ and vice versa, we can again choose transfer systems on the two and combine them, giving us the second term, $|\Tr(P\smallsetminus \{\perp\})||\Tr(Q\smallsetminus \{\perp\})|$.

Finally, assume the minimal fibrant is an element $a \in P \smallsetminus \{\top,\perp\}$. Since elements in $P \smallsetminus \{\top,\perp\}$ are never related to elements in $Q \smallsetminus \{\top,\perp\}$, we can pick a transfer system in $\Tr_a(P)$ arbitrarily. Then by restriction, $q~R~\perp$ for any $q \in Q \smallsetminus \{\top,\perp\}$, none of which are related to $\top$. Thus by \autoref{lem:botelttsbij}, we can pick $|\Tr(Q\smallsetminus \{\top,\perp\})|$ many transfer systems on this side. Summing over all possible $a$, we get the third term, $\sum_{a\in P\smallsetminus \{\top,\perp\}}^{}|\Tr_a(P)||\Tr(Q\smallsetminus \{\top,\perp\})|$. Repeating the process for a non-extremal minimal fibrant in $Q$ gives us the fourth term and thus the desired enumeration.
\end{proof}

Applying this theorem to the case of finite linear orders and invoking \cite[Theorem 20]{bbr} gives the following corollary. We write $[n] = \{0<1<\cdots<n\}$ for the finite total order on $n+1$ elements, and $\Cat(n) = \frac{1}{n+1}\binom{2n}{n}$ for the $n$-th Catalan number.

\begin{corollary}\label{cor:numbertsfusion}
For $m,n\ge 0$, the number of transfer systems on $[m]*[n]$ is
\begin{align*}
  |\Tr([m]*[n])|=2\operatorname{Cat}(n)\operatorname{Cat}(m)&+\operatorname{Cat}(n-1)(\operatorname{Cat}(m+1)-2\operatorname{Cat}(m))\\
  &+\operatorname{Cat}(m-1)(\operatorname{Cat}(n+1)-2\operatorname{Cat}(n)).
\end{align*}
In particular,
\[
  |\Tr([n]*[n])|=2\left(\operatorname{Cat}(n)^{2}+\operatorname{Cat}(n-1)(\operatorname{Cat}(n+1)-2\operatorname{Cat}(n))\right).
\]
\hfill\qedsymbol
\end{corollary}

\begin{remark}
As a final comment on categorical results, we note that $\Tr$ does not take products to products, but one may check that the canonical map $\varphi\colon \Tr(P\times Q)\to \Tr(P)\times \Tr(Q)$ is split by
\[
\begin{aligned}
  \psi:\operatorname{Tr}(P)\times\operatorname{Tr}(Q) &\longrightarrow \operatorname{Tr}(P \times Q)\\
  (R,T)&\longmapsto \{((p,q),(p',q'))\mid p~R~p' \text{ and } q~T~q'\}.
\end{aligned}
\]
Beware, though, that a similar construction does not work for fusion: there is no canonical map $\Tr(P*Q)\to \Tr(P)*\Tr(Q)$. Given the difficulty of handling colimits in lattice categories, it remains unclear to us whether $\Tr$ might be a left adjoint.
\end{remark}

\section{Transfer systems for rank two elementary Abelian groups}\label{sec:CpCp}

We can now apply our investigations to a case of interest in equivariant homotopy theory. Let $G = C_p\times C_p$ be a rank $2$ elementary Abelian $p$-group ($p$ prime). The subgroup lattice for $G$ is isomorphic to an iterated fusion of $[2] = \{0<1<2\}$ with itself as illustrated in \autoref{fig:CpCp}. This fact is well-known, but we include a proof for completeness.

\begin{lemma}\label{lem:cpcpisfusion}
The subgroup lattice of $C_p\times C_p$ is isomorphic to the $(p+1)$-fold iterated fusion of $[2]$ with itself, \emph{i.e.},
\[
  \operatorname{Sub}(C_{p}\times C_{p})\cong [2]^{*(p+1)}.
\]
\end{lemma}
\begin{proof}
Each strict nontrivial subgroup $H$ is cyclic of order $p$ and generated by any of its $p-1$ nontrivial elements of $C_p\times C_p$. Thus the total number of such subgroups is $(p^2-1)/(p-1) = p+1$.

The only covering relations in the subgroup lattice are the inclusions of the trivial subgroup $e$ into any of these $p+1$ subgroups, and the inclusion of strict nontrivial subgroups into $C_p\times C_p$. Thus the subgroup lattice is isomorphic to $[2]^{*(p+1)}$.
\end{proof}

\begin{figure}
    \centering
    \begin{tikzcd}[row sep=1.2em, column sep=1.2em]
    & & C_{p}\times C_{p}\\
    C_p \arrow[urr,dash] &
    \cdots \arrow[ur,dash] &
    \cdots \arrow[u,dash] &
    \cdots \arrow[ul,dash] &
    C_p \arrow[ull,dash]\\
    & & e \arrow[ul, dash]  \arrow[ur,dash] \arrow[ull, dash]  \arrow[urr,dash] \arrow[u,dash]
    \end{tikzcd}
    \caption{The subgroup lattice of $C_p\times C_p$}
    \label{fig:CpCp}
\end{figure}
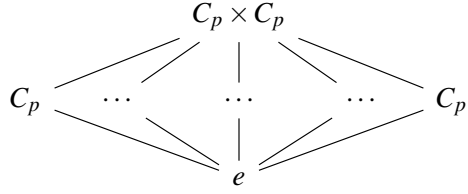

We now employ \autoref{thm:fusion} to determine the number of transfer systems on such lattices.

\begin{prop}\label{prop:interatedfusion}
For $n\ge 0$, the number of transfer systems on the $n$-fold iterated fusion of $[2]$ with itself is
\[
  |\Tr([2]^{*n})|=2^{n+1}+n
\]
where $[2]^{*0} = [1]$ and $[2]^{*1} = [2]$. Furthermore, the isomorphism type of the lattice $\Tr([2]^{*n})$ is pictured in \autoref{fig:TrCpCp} and may be described as a ``bottom'' $n$-cube $B\cong [1]^n$, a ``middle'' discrete set of $n$-elements $M$, and a ``top'' $n$-cube $T\cong [1]^n$ where the only covering relations not internal to one of the $n$-cubes are of the following forms:
\begin{enumerate}[(i)]
    \item each element of $B$ covered by $\max B$ is also covered by exactly one element of $M$,
    \item each element of $T$ covering $\min T$ also covers exactly one element of $M$,
    \item $\min T$ covers $\max B$.
\end{enumerate}
\end{prop}

\begin{figure}
    \input{figure2}
    \caption{The Hasse diagram for $\Tr([2]^{*3})\cong \Tr(C_2\times C_2)$. Compare with the Hasse diagram for saturated covers in \autoref{fig:satcovers}.}
    \label{fig:TrCpCp}
\end{figure}
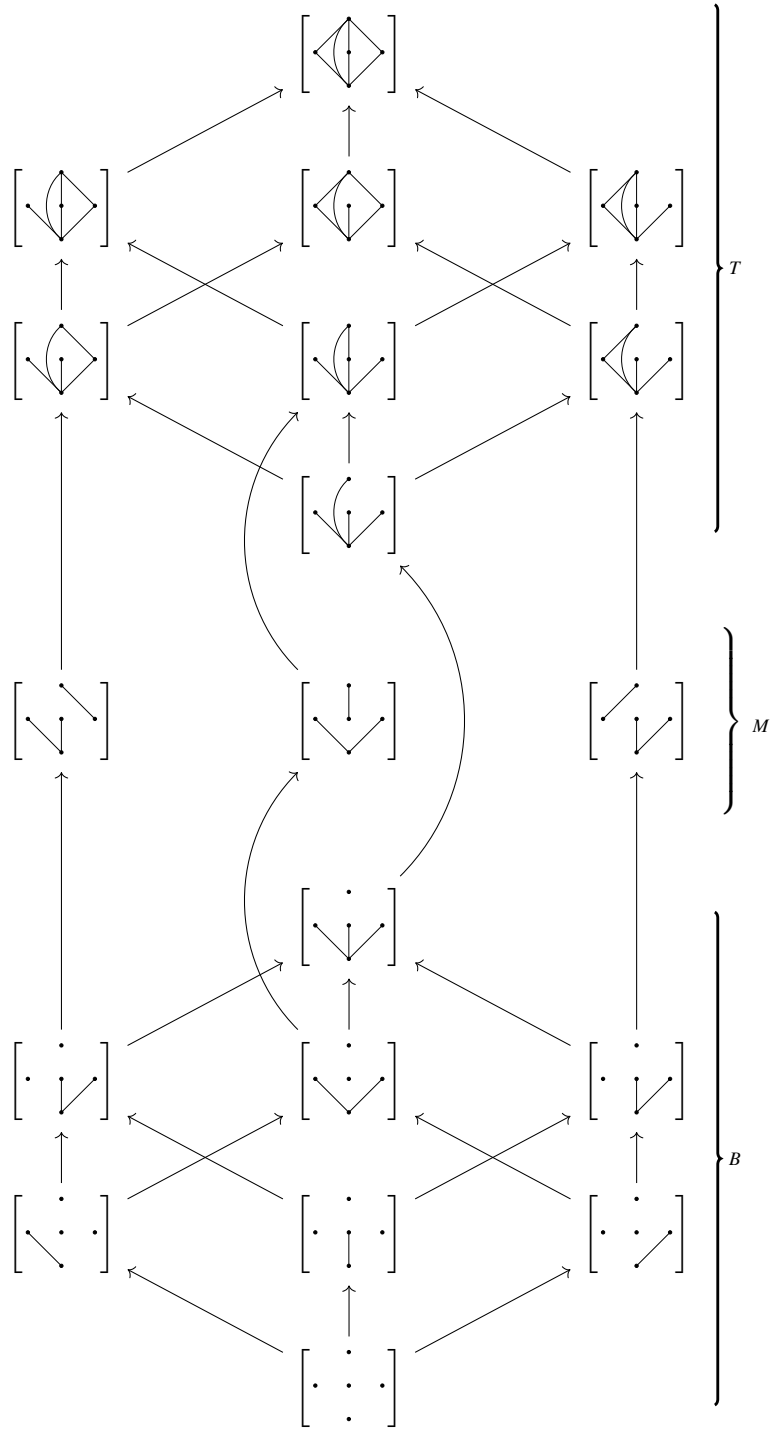

\begin{remark}
The $n=2,3$ cases of \autoref{prop:interatedfusion} appear in \cite[Figure 2]{Rubin}; the $n>3$ cases are new.
\end{remark}

\begin{proof}
For $a\in [2]^{*n}\smallsetminus \{\top,\perp\}$, note that $|\Tr_a([2]^{*n})| = 1$. Indeed, the only way a non-extremal element can be minimal fibrant for a transfer system $R$ on $[2]^{*n}$ is if $a~R~\top$ and $\perp~R~b$ for all non-extremal $b\ne a$, and there are no other non-reflexive relations. Also note that $1$ is the only non-extremal element of $[2]$, and $|\Tr([2]^{*n}\smallsetminus\{\top,\perp\})| = 1$, $|\Tr_1([2])| = 1$.

These observations allow us to specialize \autoref{thm:fusion} to get
\[
\begin{aligned}
    |\Tr([2]^{*(n+1)}| &= |\Tr([2]^{*n}*[2])| \\
    &= 2|\Tr([2]^{*n}\smallsetminus \{\top\}|+2|\Tr([2]^{*n}\smallsetminus \{\perp\})| + n+1\\
    &= 2\cdot 2^n + 2\cdot 2^n + n + 1\\
    &= 2^{n+2}+n+1
\end{aligned}
\]
where the third line holds by observing that there are $n$ independent choices of non-reflexive relations for a transfer system on $[2]^{*n}$ without one of its extremal elements. For $n\ge 1$ this is equivalent to the formula in the proposition, and small values follow by inspection.

It remains to specify the lattice structure of $\Tr([2]^{*n})$. The bottom cube $B$ consists of transfer systems whose non-reflexive relations are a subset of relations of the form $\perp~R~a$ where $a\ne \perp,\top$. The top cube $T$ consists of transfer systems with all relations $\perp~R~a$, the relation $\perp~R~\top$, and a subset of relations $a~R~\top$ (where $a\ne \perp,\top$ throughout). The intermediate transfer systems $M$ have exactly one relation $a~R~\top$ (for $a\ne \perp,\top$) and all relations $\perp~R~b$ for $b\ne a,\perp,\top$.

The sets $B,T,M$ are clearly disjoint and $|B\cup T\cup M| = 2^n+2^n+n = 2^{n+1}+n$, so this specifies all elements of $\Tr([2]^{*n})$. The covering relations (i)--(iii) now follow directly from our description of the transfer systems.
\end{proof}

The main theorem of this section is now an immediate corollary of \autoref{lem:cpcpisfusion} and \autoref{prop:interatedfusion}.

\begin{theorem}\label{thm:ranktwo}
Let $G = C_p\times C_p$ for $p$ a prime number. Then there are exactly $2^{p+2}+p+1$ $G$-transfer systems, and the lattice $\Tr(G)$ takes the form described in \autoref{prop:interatedfusion} with $n=p+1$.\hfill\qedsymbol
\end{theorem}

\begin{example}
We may also illustrate the action of Hill's characteristic map on $\Tr([2]^{*n})$ and highlight the saturated transfer systems (modular matchstick games) within this structure. First note that all elements of $B$ and $M$ are both saturated and are minimal elements of fibers of $\chi$ according to \autoref{thm:minchi}. As such, these transfer systems represent the singleton fibers of $\chi\colon \Tr([2]^{*n})\to \End^\circ([2]^{*n})$. The only additional fiber of $\chi$ is the top cube $T$, which is the preimage of the interior operator taking the constant value $\perp$. In particular, there are exactly $2^n+n+1$ saturated transfer systems on $[2]^{*n}$, and this is also the number of interior operators on $[2]^{*n}$. As an antitone function, $\chi\colon \Tr([2]^{*n})\to \End^\circ([2]^{*n})$ collapses $T$ to the minimal interior operator (constant on $\perp$) and is an order-reversing bijection away from $T$.
\end{example}

\bibliography{chi_matchsticks}
\bibliographystyle{alpha}

\end{document}

%% file: Tamari2.tex
\newcommand{\braktwo}{
        \coordinate (000) at (0:0); 
        \coordinate (010) at (90:{sqrt(2)/2}); 
        \coordinate (111) at ($(010) + (010)$);
    
        \foreach \i in {(000),(010),(111)}{
            \draw[white, fill =white] \i circle (.1);
            \draw[black,fill = black] \i circle (.035);
        }
}

\newcommand{\trlppppp}[2]{
    \draw[line width=0.1mm] (#1)--(#2);
}
\newcommand{\trcvv}[2]{
    \draw[line width=0.1mm] (#1) to[out=130,in=-130] (#2);
}

\newcommand{\betikzzzz}[1]{
    \begin{tikzpicture}[scale = 0.75]
        #1
    \end{tikzpicture}
}

\begin{center}
    \begin{tikzcd}[row sep=-0.2em, column sep=2em]
    &
    \begin{bmatrix}\betikzzzz{\braktwo\trlppppp{000}{010}\trlppppp{010}{111}\trcvv{000}{111}} \end{bmatrix}\\
    \begin{bmatrix}\betikzzzz{\braktwo\trlppppp{000}{010}\trcvv{000}{111}} \end{bmatrix}\arrow[ur]\\
    &
    &
    \begin{bmatrix}\betikzzzz{\braktwo\trlppppp{010}{111}} \end{bmatrix} \arrow[uul]\\
    \begin{bmatrix}\betikzzzz{\braktwo\trlppppp{000}{010}} \end{bmatrix} \arrow[uu]\\
    &
    \begin{bmatrix}\betikzzzz{\braktwo} \end{bmatrix} \arrow[uur] \arrow[ul]
    \end{tikzcd}
\end{center}

%% file: saturated_covers.tex
\newcommand{\fusion}{
        \coordinate (000) at (0:0); 
        \coordinate (100) at (135:1); 
        \coordinate (001) at (45:1); 
        \coordinate (010) at (90:{sqrt(2)/2}); 
        \coordinate (111) at ($(100) + (001)$);
    
        \foreach \i in {(000),(100),(010),(001),(111)}{
            \draw[white, fill =white] \i circle (.1);
            \draw[black,fill = black] \i circle (.035);
        }
}

\newcommand{\trlp}[2]{
    \draw[line width=0.1mm] (#1)--(#2);
}

\newcommand{\betikzz}[1]{
    \begin{tikzpicture}[scale = 0.85]
        #1
    \end{tikzpicture}
}

\begin{center}
\begin{tikzcd}[row sep=4em, column sep=6em]
    & 
    \begin{bmatrix}\betikzz{\fusion \trlp{000}{010}\trlp{010}{111}\trlp{000}{001}\trlp{100}{111}\trlp{000}{100}\trlp{001}{111}}\end{bmatrix}\\
    \begin{bmatrix}\betikzz{\fusion \trlp{000}{010}\trlp{000}{100}\trlp{001}{111}}\end{bmatrix} \arrow[ur]&
    \begin{bmatrix}\betikzz{\fusion \trlp{000}{100}\trlp{000}{001}\trlp{010}{111}}\end{bmatrix}\arrow[u] &
    \begin{bmatrix}\betikzz{\fusion \trlp{000}{010}\trlp{000}{001}\trlp{100}{111}}\end{bmatrix} \arrow[ul]\\
    & 
    \begin{bmatrix}\betikzz{\fusion \trlp{000}{001}\trlp{000}{100}\trlp{000}{010}}\end{bmatrix} \arrow[uu,bend left=35]\\
    \begin{bmatrix}\betikzz{\fusion \trlp{000}{001}\trlp{000}{010}}\end{bmatrix} \arrow[uu] \arrow[ur] &
    \begin{bmatrix}\betikzz{\fusion \trlp{000}{001}\trlp{000}{100}}\end{bmatrix}\arrow[u] \arrow[uu,bend right=35]&
    \begin{bmatrix}\betikzz{\fusion \trlp{000}{010}\trlp{000}{001}}\end{bmatrix} \arrow[uu] \arrow[ul]\\
    \begin{bmatrix}\betikzz{\fusion \trlp{000}{100}}\end{bmatrix} \arrow[ur] \arrow[u]&
    \begin{bmatrix}\betikzz{\fusion \trlp{000}{010}}\end{bmatrix} \arrow[ul] \arrow[ur]&
    \begin{bmatrix}\betikzz{\fusion \trlp{000}{001}}\end{bmatrix} \arrow[ul] \arrow[u]\\
    &
    \begin{bmatrix}\betikzz{\fusion}\end{bmatrix} \arrow[ur] \arrow[u] \arrow[ul]
\end{tikzcd}
\end{center}

%% file: saturated_diagrams.tex
\newcommand{\squaresetup}{
    \coordinate (00) at (0,0);
    \coordinate (10) at (135:1);
    \coordinate (01) at (45:1);
    \coordinate (11) at ($(10)+(01)$);
    
    \draw[gray] (00) -- (10) -- (11) -- (01)--(00);
    \foreach \i in {(00),(01),(10),(11)}{
        \draw[white,fill = white] \i circle (.1);
        \draw[gray,fill = gray] \i circle (.05);
    }
    \draw (0,2) node {}; 
}

\newcommand{\cube}{ 
        \coordinate (000) at (0:0); 
        \coordinate (100) at (135:1); 
        \coordinate (001) at (45:1); 
        \coordinate (010) at (90:{sqrt(2)/2}); 
        \coordinate (110) at ($(100) + (010)$);
        \coordinate (101) at ($(100) + (001)$);
        \coordinate (011) at ($(010) + (001)$);
        \coordinate (111) at ($(100) + (010) + (001)$);

        \draw[gray] (000) -- (001);
        \draw[gray] (000) -- (010);
        \draw[gray] (000) -- (100);
        \draw[gray] (100) -- (110);
        \draw[gray] (100) -- (101);
        \draw[gray] (001) -- (101);
        \draw[gray] (001) -- (011);
        \draw[white, ultra thick] (010) -- (110);
        \draw[white, ultra thick] (010) -- (011);
        \draw[gray] (010) -- (110);
        \draw[gray] (010) -- (011);
        \draw[gray] (111) -- (011);
        \draw[gray] (111) -- (101);
        \draw[gray] (111) -- (110);

        \foreach \i in {(000),(100),(010),(001),(011),(101),(110),(111)}{
            \draw[white, fill =white] \i circle (.1);
            \draw[gray,fill = gray] \i circle (.05);
        }
}

\newcommand{\tlp}[2]{
    \draw[ultra thick,line cap =round] (#1)--(#2);
}

\newcommand{\betikz}[1]{ 
    \begin{tikzpicture}[scale = .5]
        #1
    \end{tikzpicture}
}

\begin{center}
    \begin{longtable}{c|cccccc}
    \caption{Saturated covers for $[1]^3$ organized according to bottom layer. Bolded covering relations are those included in each saturated cover.}\label{tab:satcov}\\
        Bottom &  \multicolumn{6}{c}{Saturated covers with this bottom layer}\\\hline
        \betikz{\squaresetup}  & 
        \betikz{\cube}&
        \betikz{\cube \tlp{000}{010}}&
        \betikz{\cube\tlp{100}{110}\tlp{000}{010}}&
        \betikz{\cube\tlp{001}{011}\tlp{000}{010}}&
        \betikz{\cube\tlp{001}{011}\tlp{100}{110}\tlp{000}{010}}&
        \betikz{\cube\tlp{001}{011}\tlp{100}{110}\tlp{000}{010}\tlp{101}{111}}
        \\\hline

        \betikz{\squaresetup \tlp{00}{01}}& 
        \betikz{\cube\tlp{000}{001}} &
        \betikz{\cube\tlp{000}{001}\tlp{000}{010}}&
        \betikz{\cube\tlp{000}{001}\tlp{000}{010}\tlp{100}{110}} &&&
        \\ &
        \betikz{\cube\tlp{000}{001}\tlp{011}{010}} &&&
        \betikz{\cube\tlp{000}{001}\tlp{011}{010}\tlp{000}{010}\tlp{001}{011}} &
        \betikz{\cube\tlp{000}{001}\tlp{011}{010}\tlp{000}{010}\tlp{001}{011}\tlp{100}{110}}&
        \betikz{\cube\tlp{000}{001}\tlp{011}{010}\tlp{000}{010}\tlp{001}{011}\tlp{100}{110}\tlp{101}{111}}     
        \\\hline

        \betikz{\squaresetup \tlp{00}{10}}&
        \betikz{\cube\tlp{000}{100}} &
        \betikz{\cube\tlp{000}{100}\tlp{000}{010}}&&
        \betikz{\cube\tlp{000}{100}\tlp{000}{010}\tlp{001}{011}} &&
        \\ &
        \betikz{\cube\tlp{000}{100}\tlp{110}{010}} &&
        \betikz{\cube\tlp{000}{100}\tlp{110}{010}\tlp{000}{010}\tlp{100}{110}}&&
        \betikz{\cube\tlp{000}{100}\tlp{110}{010}\tlp{000}{010}\tlp{100}{110}\tlp{001}{011}}&
        \betikz{\cube\tlp{000}{100}\tlp{110}{010}\tlp{000}{010}\tlp{100}{110}\tlp{001}{011}\tlp{101}{111}} \\\hline

        \betikz{\squaresetup \tlp{00}{10}\tlp{00}{01}}&
        \betikz{\cube \tlp{001}{000}\tlp{100}{000}}& 
        \betikz{\cube \tlp{001}{000}\tlp{100}{000}\tlp{000}{010}}&
        \\ &
        \betikz{\cube \tlp{001}{000}\tlp{100}{000}\tlp{110}{010}}&&
        \betikz{\cube \tlp{001}{000}\tlp{100}{000}\tlp{110}{010}\tlp{100}{110}\tlp{000}{010}}
        \\ &
        \betikz{\cube\tlp{001}{000}\tlp{100}{000}\tlp{010}{011}}&&&
        \betikz{\cube\tlp{001}{000}\tlp{100}{000}\tlp{010}{011}\tlp{000}{010}\tlp{001}{011}}
        \\ &
        \betikz{\cube\tlp{001}{000}\tlp{100}{000}\tlp{010}{110}\tlp{010}{011}}&&&&
        \betikz{\cube\tlp{001}{000}\tlp{100}{000}\tlp{010}{110}\tlp{010}{011}\tlp{000}{010}\tlp{100}{110}\tlp{001}{011}}&
        \betikz{\cube\tlp{001}{000}\tlp{100}{000}\tlp{010}{110}\tlp{010}{011}\tlp{000}{010}\tlp{100}{110}\tlp{001}{011}\tlp{101}{111}}

        \\\hline
        \betikz{\squaresetup \tlp{00}{01}\tlp{10}{11}}&
        \betikz{\cube\tlp{000}{001}\tlp{100}{101}}&
        \betikz{\cube\tlp{000}{001}\tlp{100}{101}\tlp{000}{010}}&
        \betikz{\cube\tlp{000}{001}\tlp{100}{101}\tlp{000}{010}\tlp{100}{110}}&
        \\ &
        \betikz{\cube\tlp{000}{001}\tlp{100}{101}\tlp{010}{011}}&&&
        \betikz{\cube\tlp{000}{001}\tlp{100}{101}\tlp{010}{011}\tlp{000}{010}\tlp{001}{011}}&
        \betikz{\cube\tlp{000}{001}\tlp{100}{101}\tlp{010}{011}\tlp{000}{010}\tlp{001}{011}\tlp{100}{110}}
        \\ &
        \betikz{\cube\tlp{000}{001}\tlp{100}{101}\tlp{010}{011}\tlp{110}{111}}&&&
        \betikz{\cube\tlp{000}{001}\tlp{100}{101}\tlp{010}{011}\tlp{110}{111}\tlp{000}{010}\tlp{001}{011}}&&
        \betikz{\cube\tlp{000}{001}\tlp{100}{101}\tlp{010}{011}\tlp{110}{111}\tlp{000}{010}\tlp{001}{011}\tlp{100}{110}\tlp{101}{111}}

        \\\hline
        \betikz{\squaresetup \tlp{00}{10}\tlp{01}{11}}&
        \betikz{\cube\tlp{000}{100}\tlp{001}{101}}&
        \betikz{\cube\tlp{000}{100}\tlp{001}{101}\tlp{000}{010}}&&
        \betikz{\cube\tlp{000}{100}\tlp{001}{101}\tlp{000}{010}\tlp{001}{011}}
        \\ &
        \betikz{\cube\tlp{000}{100}\tlp{001}{101}\tlp{010}{110}} &&
        \betikz{\cube\tlp{000}{100}\tlp{001}{101}\tlp{010}{110}\tlp{000}{010}\tlp{100}{110}} && 
        \betikz{\cube\tlp{000}{100}\tlp{001}{101}\tlp{010}{110}\tlp{000}{010}\tlp{100}{110}\tlp{001}{011}}
        \\ &
        \betikz{\cube\tlp{000}{100}\tlp{001}{101}\tlp{010}{110}\tlp{011}{111}} &&
        \betikz{\cube\tlp{000}{100}\tlp{001}{101}\tlp{010}{110}\tlp{011}{111}\tlp{000}{010}\tlp{100}{110}} &&&
        \betikz{\cube\tlp{000}{100}\tlp{001}{101}\tlp{010}{110}\tlp{011}{111}\tlp{000}{010}\tlp{100}{110}\tlp{001}{011}\tlp{101}{111}}

        \\\hline
        \betikz{\squaresetup \tlp{00}{01}\tlp{10}{11}\tlp{01}{11}\tlp{00}{10}}& 
        \betikz{\cube\tlp{000}{001}\tlp{000}{100}\tlp{100}{101}\tlp{001}{101}}&
        \betikz{\cube\tlp{000}{001}\tlp{000}{100}\tlp{100}{101}\tlp{001}{101}\tlp{000}{010}}
        \\ &
        \betikz{\cube\tlp{000}{001}\tlp{000}{100}\tlp{100}{101}\tlp{001}{101}\tlp{010}{110}} &&
        \betikz{\cube\tlp{000}{001}\tlp{000}{100}\tlp{100}{101}\tlp{001}{101}\tlp{010}{110}\tlp{000}{010}\tlp{100}{110}}
        \\ &
        \betikz{\cube\tlp{000}{001}\tlp{000}{100}\tlp{100}{101}\tlp{001}{101}\tlp{010}{011}} &&&
        \betikz{\cube\tlp{000}{001}\tlp{000}{100}\tlp{100}{101}\tlp{001}{101}\tlp{010}{011}\tlp{000}{010}\tlp{001}{011}}
        \\ &
        \betikz{\cube\tlp{000}{001}\tlp{000}{100}\tlp{100}{101}\tlp{001}{101}\tlp{010}{110}\tlp{010}{011}} &&&&
        \betikz{\cube\tlp{000}{001}\tlp{000}{100}\tlp{100}{101}\tlp{001}{101}\tlp{010}{110}\tlp{010}{011}\tlp{100}{110}\tlp{000}{010}\tlp{001}{011}}
        \\ &
        \betikz{\cube\tlp{000}{001}\tlp{000}{100}\tlp{100}{101}\tlp{001}{101}\tlp{010}{110}\tlp{011}{111}} &&
        \betikz{\cube\tlp{000}{001}\tlp{000}{100}\tlp{100}{101}\tlp{001}{101}\tlp{010}{110}\tlp{011}{111}\tlp{000}{010}\tlp{100}{110}}
        \\ &
        \betikz{\cube\tlp{000}{001}\tlp{000}{100}\tlp{100}{101}\tlp{001}{101}\tlp{010}{011}\tlp{110}{111}} &&&
        \betikz{\cube\tlp{000}{001}\tlp{000}{100}\tlp{100}{101}\tlp{001}{101}\tlp{010}{011}\tlp{110}{111}\tlp{000}{010}\tlp{001}{011}}
        \\ &
        \betikz{\cube\tlp{000}{001}\tlp{000}{100}\tlp{100}{101}\tlp{001}{101}\tlp{010}{011}\tlp{010}{110}\tlp{110}{111}\tlp{011}{111}} &&&&&
        \betikz{\cube\tlp{000}{001}\tlp{000}{100}\tlp{100}{101}\tlp{001}{101}\tlp{010}{011}\tlp{010}{110}\tlp{110}{111}\tlp{011}{111}\tlp{000}{010}\tlp{100}{110}\tlp{001}{011}\tlp{101}{111}}
    \end{longtable}
\end{center}

%% file: figure2.tex
\newcommand{\trcv}[2]{
    \draw[line width=0.1mm] (#1) to[out=140,in=-140] (#2);
}

\newcommand{\fusion}{
        \coordinate (000) at (0:0); 
        \coordinate (100) at (135:1); 
        \coordinate (001) at (45:1); 
        \coordinate (010) at (90:{sqrt(2)/2}); 
        \coordinate (111) at ($(100) + (001)$);
    
        \foreach \i in {(000),(100),(010),(001),(111)}{
            \draw[white, fill =white] \i circle (.1);
            \draw[black,fill = black] \i circle (.035);
        }
}

\newcommand{\trlp}[2]{
    \draw[line width=0.1mm] (#1)--(#2);
}

\newcommand{\betikzz}[1]{
    \begin{tikzpicture}[scale = 0.7]
        #1
    \end{tikzpicture}
}

\begin{center}
\adjustbox{scale=0.9,center}{
    \begin{tikzcd}[row sep=2em, column sep=6em]
    & 
    \begin{bmatrix}\betikzz{\fusion \trlp{000}{010}\trlp{010}{111}\trlp{000}{001}\trlp{100}{111}\trlp{000}{100}\trlp{001}{111}\trcv{000}{111}}\end{bmatrix}&\vphantom{\Bigg|}\ar[ddd, start anchor = north, end anchor = south,xshift=3em,no head,decorate,decoration = {brace},very thick,"T" right = 3pt]\\

    \begin{bmatrix}\betikzz{\fusion \trlp{000}{010}\trlp{000}{100}\trlp{001}{111}\trlp{000}{001}\trlp{010}{111}\trcv{000}{111}}\end{bmatrix} \arrow[ur]&
    \begin{bmatrix}\betikzz{\fusion \trlp{000}{100}\trlp{000}{001}\trlp{000}{010}\trlp{100}{111}\trlp{001}{111}\trcv{000}{111}}\end{bmatrix}\arrow[u] &
    \begin{bmatrix}\betikzz{\fusion \trlp{000}{010}\trlp{000}{001}\trlp{100}{111}\trlp{000}{100}\trlp{010}{111}\trcv{000}{111}}\end{bmatrix} \arrow[ul]\\

    \begin{bmatrix}\betikzz{\fusion \trlp{000}{010}\trlp{000}{100}\trlp{001}{111}\trlp{000}{001}\trcv{000}{111}}\end{bmatrix} \arrow[ur] \arrow[u]&
    \begin{bmatrix}\betikzz{\fusion \trlp{000}{100}\trlp{000}{001}\trlp{010}{111}\trlp{000}{010}\trcv{000}{111}}\end{bmatrix}\arrow[ur] \arrow[ul] &
    \begin{bmatrix}\betikzz{\fusion \trlp{000}{010}\trlp{000}{001}\trlp{100}{111}\trlp{000}{100}\trcv{000}{111}}\end{bmatrix} \arrow[ul] \arrow[u]\\

    & 
    \begin{bmatrix}\betikzz{\fusion \trlp{000}{001}\trlp{000}{100}\trlp{000}{010}\trcv{000}{111}}\end{bmatrix} \arrow[ul] \arrow[u] \arrow[ur]&\vphantom{Mid}\\

    \begin{bmatrix}\betikzz{\fusion \trlp{000}{010}\trlp{000}{100}\trlp{001}{111}}\end{bmatrix} \arrow[uu] &
    \begin{bmatrix}\betikzz{\fusion \trlp{000}{100}\trlp{000}{001}\trlp{010}{111}}\end{bmatrix} \arrow[uu,bend left=45]&
    \begin{bmatrix}\betikzz{\fusion \trlp{000}{010}\trlp{000}{001}\trlp{100}{111}}\end{bmatrix} \arrow[uu]&[-70pt]\left.\begin{matrix}\textcolor{white}{0}\\\textcolor{white}{0}\\\textcolor{white}{0}\\\textcolor{white}{0}\\\textcolor{white}{0}\\\textcolor{white}{0}\end{matrix}\right\}\!\!\textcolor{white}{a}_M\\ 

    & 
    \begin{bmatrix}\betikzz{\fusion \trlp{000}{001}\trlp{000}{100}\trlp{000}{010}}\end{bmatrix} \arrow[uu,bend right=45] & \vphantom{a}\ar[ddd, start anchor = north, end anchor = south,xshift=3em,no head,decorate,decoration = {brace},"B" right = 3pt,very thick]\\

    \begin{bmatrix}\betikzz{\fusion \trlp{000}{001}\trlp{000}{010}}\end{bmatrix} \arrow[uu] \arrow[ur] &
    \begin{bmatrix}\betikzz{\fusion \trlp{000}{001}\trlp{000}{100}}\end{bmatrix}\arrow[u] \arrow[uu,bend left=45]&
    \begin{bmatrix}\betikzz{\fusion \trlp{000}{010}\trlp{000}{001}}\end{bmatrix} \arrow[uu] \arrow[ul]\\

    \begin{bmatrix}\betikzz{\fusion \trlp{000}{100}}\end{bmatrix} \arrow[ur] \arrow[u]&
    \begin{bmatrix}\betikzz{\fusion \trlp{000}{010}}\end{bmatrix} \arrow[ul] \arrow[ur]&
    \begin{bmatrix}\betikzz{\fusion \trlp{000}{001}}\end{bmatrix} \arrow[ul] \arrow[u]\\

    &
    \begin{bmatrix}\betikzz{\fusion}\end{bmatrix} \arrow[ur] \arrow[u] \arrow[ul] & \vphantom{ASDF}
\end{tikzcd}
}
\end{center}